\documentclass{amsart}

\usepackage{amssymb}
\usepackage{amsthm}
\usepackage{amsmath}
\usepackage{enumitem}
\usepackage[mathscr]{euscript}
\usepackage{mathtools}
\usepackage{stmaryrd}
\usepackage{cancel}
\usepackage{soul}
\usepackage{tikz}
\usepackage{tikz-cd}
\usepackage[normalem]{ulem}
\usepackage[unicode]{hyperref}
\hypersetup{colorlinks,linkcolor=blue}
\usepackage{cleveref}
\usetikzlibrary{fadings}
\usepackage{float}
\usetikzlibrary{patterns}
\usetikzlibrary{shadows.blur}
\usetikzlibrary{shapes}
\usepackage{xcolor}

\usepackage{todonotes}
\usepackage{comment}
\usepackage{eqparbox}

\theoremstyle{plain}

\newtheorem{theoremalpha}{Theorem}

\newtheorem{proposition}{Proposition}[section]
\newtheorem{theorem}[proposition]{Theorem}
\newtheorem{lemma}[proposition]{Lemma}
\newtheorem{corollary}[proposition]{Corollary}
\theoremstyle{definition}
\newtheorem{example}[proposition]{Example}
\newtheorem{definition}[proposition]{Definition}
\newtheorem{observation}[proposition]{Observation}
\theoremstyle{remark}
\newtheorem{remark}[proposition]{Remark}
\newtheorem{notation}[proposition]{Notation}

\DeclareMathOperator{\Aut}{Aut}

\DeclareMathOperator{\diam}{diam}

\DeclareMathOperator{\supp}{supp}

\DeclareMathOperator{\Cay}{Cay}
\DeclareMathOperator{\Curr}{Curr}

\DeclareMathOperator{\SO}{\mathsf{SO}}

\DeclareMathOperator{\Hom}{Hom}
\DeclareMathOperator{\Out}{Out}

\DeclareMathOperator{\Teich}{Teich}

\DeclareMathOperator{\id}{id} 
 
\DeclareMathOperator{\Haus}{Haus} 

\DeclareMathOperator{\Mod}{Mod}
\DeclareMathOperator{\Isom}{Isom}

\DeclareMathOperator{\Gal}{Gal} 
\DeclareMathOperator{\Graph}{Graph}
\DeclareMathOperator{\PML}{\mathbb{P}\mathcal{ML}}
\DeclareMathOperator{\PCurr}{\mathbb{P}\mathrm{Curr}}
\DeclareMathOperator{\Axis}{Axis}

\DeclareMathOperator{\Ac}{\mathcal{A}}
\DeclareMathOperator{\Bc}{\mathcal{B}}
\DeclareMathOperator{\Cc}{\mathcal{C}}
\DeclareMathOperator{\Ec}{\mathcal{E}}

\DeclareMathOperator{\Gc}{\mathcal{G}}
\DeclareMathOperator{\Hc}{\mathcal{H}}

\DeclareMathOperator{\Lc}{\mathcal{L}}
\DeclareMathOperator{\MM}{\mathcal{M}}
\DeclareMathOperator{\Nc}{\mathcal{N}}

\DeclareMathOperator{\Pc}{\mathcal{P}}

\DeclareMathOperator{\Tc}{\mathcal{T}}
\DeclareMathOperator{\Uc}{\mathcal{U}}

\DeclareMathOperator{\Xc}{\mathcal{X}}

\DeclareMathOperator{\Hb}{\mathbb{H}}

\DeclareMathOperator{\Nb}{\mathbb{N}}

\DeclareMathOperator{\Rb}{\mathbb{R}}
\DeclareMathOperator{\Sb}{\mathbb{S}}

\DeclareMathOperator{\Zb}{\mathbb{Z}}
\DeclareMathOperator{\Qb}{\mathbb{Q}}

\DeclareMathOperator{\Y}{\Upsilon}

\DeclareMathOperator{\Ebdd}{\Ec_{bdd}}
\DeclareMathOperator{\Econt}{\Ec_{cont}}

\DeclareMathOperator{\Eesc}{\Ec_{esc}}
\DeclareMathOperator{\Emesc}{\Ec_{esc,m}}
\DeclareMathOperator{\Ehausesc}{\Ec_{esc,h}}
\newcommand{\Epmesc}{\Ec'_\mathrm{esc,m}}

\DeclareMathOperator{\hEbdd}{\widehat{\Ec}_{bdd}}
\DeclareMathOperator{\hEcont}{\widehat{\Ec}_{cont}}
\DeclareMathOperator{\hEmcont}{\widehat{\Ec}_{cont,m}}

\DeclareMathOperator{\hEesc}{\widehat{\Ec}_{esc}}
\DeclareMathOperator{\hEmesc}{\widehat{\Ec}_{esc,m}}

\newcommand{\hEpmesc}{\widehat{\Ec'}_\mathrm{esc,m}}

\newcommand{\de}{\mathrm{d}}
\newcommand{\abs}[1]{\left|#1\right|}

\newcommand{\norm}[1]{\left\|#1\right\|}

\newcommand{\wh}[1]{\widehat{#1}}
\newcommand{\ip}[1]{\left\langle #1\right\rangle}
\renewcommand{\epsilon}{\varepsilon}

\begin{document}

\title[Ending currents of hyperbolic group actions]{Ending currents of hyperbolic group actions}

\author{Yair N. Minsky}
\address{Department of Mathematics, Yale University}
\email{yair.minsky@yale.edu}
\thanks{}

\author{Tianqi Wang}
\address{Department of Mathematics, Yale University}
\email{tq.wang@yale.edu}
\thanks{}

\date{\today}

\begin{abstract}  
    We develop a theory of ending currents for properly discontinuous isometric actions of hyperbolic groups on Gromov hyperbolic spaces, generalizing Thurston's ending laminations. Several natural generalizations are considered,	which agree in the 3-dimensional hyperbolic setting but are otherwise distinct. We illustrate these invariants and their differences by several examples, including closed surface group actions on hyperbolic group extensions arising from Teichm\"uller theory and a geometrically finite action of a closed surface group on $\Hb^4$ with a single cusp that corresponds to a non-simple curve on the surface. We also discuss the conditions under which ending currents characterize geometrically finite actions, obtaining several equivalent conditions and some surprises, illustrated by an example of a geometrically infinite action whose ending currents are all supported on a single orbit of peripheral subgroups.
\end{abstract}

\maketitle
\setcounter{tocdepth}{1}
\tableofcontents

\section{Introduction}

Thurston's ending laminations, introduced in the late 1970s, encode the failure of a Kleinian group to be convex-cocompact. In this setting, these laminations simultaneously capture a number of a priori distinct ``degeneration'' phenomena for a discrete faithful representation $\rho:\Gamma \to G$, where $\Gamma$ is the fundamental group of a 3-manifold and $G$ is the isometry group of hyperbolic $3$-space:
\begin{enumerate}
    \item {\em Distortion}, the vanishing of \[\inf_{\gamma\in \Gamma\setminus\{\id\}} \frac{\ell_\rho(\gamma)}{\norm{\gamma}}\] where $\norm{\gamma}$ denotes stable translation length of $\gamma$ on some Cayley graph for $\Gamma$, and $\ell_\rho$ denotes stable translation length in the image.
    \item {\em Escape}, the ability of geodesic representatives of elements of $\Gamma$ to exit any compact set in the image manifold.
    \item {\em Collapse at infinity}, which is the failure of an associated Cannon--Thurston map (when one exists) to be injective.
\end{enumerate}
Each of these phenomena leads to a limiting process (where the type of limit taken is also negotiable), and the limiting object is the same for all of them. In particular it is represented by a {\em lamination}, a system of non-intersecting geodesics supported on the boundary of the corresponding 3-manifold. It plays a central role in the study of hyperbolic 3-manifolds and their deformation spaces (e.g. \cite{minsky2010classification,brock2012classification}).

One can consider this in a variety of more general settings where $\Gamma$ is a finitely generated hyperbolic group, and $G$ is the isometry group of some $\delta$-hyperbolic space $X$. Bowditch \cite{bowditch:markoff} considered the case where $G$ is still $\Isom(\Hb^3)$ but $\rho$ is not necessarily discrete. Mj \cite{mitra1997ending,mitra1998cannon,mitra1999theorem} and Mj--Rafi \cite{mj2018algebraic} considered normal subgroups $\Gamma$ of a $\delta$-hyperbolic group $G$, and particularly interesting cases of this arise for mapping class groups through Birman's exact sequence by work of Farb--Mosher \cite{farb2002convex} and Hamenst\"adt \cite{hamenstadt2005word}. See Section~\ref{introduction section: hyperbolic group extensions} for details on this.

The tight correlation between the different degeneration phenomena in dimension $3$ breaks down in more general contexts. In this paper we give a systematic study of several different definitions of ``ending current'' that capture these phenomena, develop the relations between them and through examples show some of the distinct behaviors. 

Even when replacing $\Hb^3$ by a higher-dimensional hyperbolic space, the non-intersecting property that leads the invariant to be a lamination fails. We illustrate this in the following result, which is closely related to a theorem of Lee \cite{Lee11}.

\begin{theoremalpha}\label{main theorem: the example of simple-stable representations}
    For any closed oriented hyperbolic surface $\Sigma$, there exists a geometrically finite representation $\rho_0:\pi_1(\Sigma)\to\Isom^+(\Hb^4)$ with a single conjugacy class of maximal parabolic subgroups, generated by a non-simple closed geodesic on $\Sigma$.
\end{theoremalpha}

For a general $\delta$-hyperbolic group $\Gamma$, the notion of intersection of lines is no longer natural anyway, and so we consider the space of {\em geodesic currents of a group}:

Denote by $\partial_\infty^{(2)}\Gamma$ the space of ordered pairs of distinct points in $\partial_\infty\Gamma$, where $\partial_\infty\Gamma$ is the Gromov boundary. We let $\Curr(\Gamma)$ denote Bonahon's space of {\emph geodesic currents} on $\Gamma$, namely the space of locally finite Borel measures on $\partial_\infty^{(2)}\Gamma$ which are invariant by $\Gamma$ and by the involution in $\partial_\infty^{(2)}\Gamma$, endowed with the weak-$*$ topology. Let $\PCurr(\Gamma) = (\Curr(\Gamma)\setminus\{0\} )/\Rb_{>0}$, the compact space of {\em projectivized currents.}

Throughout the paper, we consider a hyperbolic group $\Gamma$ acting properly discontinuously by isometries on a proper geodesic $\delta$-hyperbolic space $X$ via a representation \[\rho:\Gamma\to \Isom(X)~.\]

\begin{remark}
    For convenience, we assume if $\rho(\gamma)$ is loxodromic, it admits a $\rho(\gamma)$-invariant geodesic axis. Whenever we refer to $\Axis_{\rho(\gamma)}$, the axis of $\rho(\gamma)$, we mean such an invariant geodesic. Throughout the paper, this assumption can be removed by working with quasi-geodesic axes up to enlarge the constants.
\end{remark}

For each $\gamma\in \Gamma$ of infinite order, we denote by $\widehat{\gamma}\in \Curr(\Gamma)$ the counting measure of the $\Gamma$-orbit of the pair of endpoints of the axis of $\gamma$. (We also sometimes use $\widehat\gamma$ to denote the image in $\PCurr(\Gamma)$).

Fixing a base point $o\in X$, let $\tau_{\rho,o}:\Gamma\to X$ be the orbit map $\tau_{\rho,o}(\gamma) = \rho(\gamma)o$ and let $p:X\to \rho(\Gamma)\backslash X$ be the natural projection. For $\gamma\in \Gamma$, if $\rho(\gamma)$ is loxodromic, we let $r(\gamma)$ be the distance from $p(o)$ to the $p$-image of (some) axis of $\rho(\gamma)$ in $\rho(\Gamma)\backslash X$, and if $\rho(\gamma)$ is parabolic, we set $r(\gamma)=+\infty$. If $\rho(\gamma)$ is elliptic or the identity, we let $r(\gamma)=0$.

A {\em weighted multigeodesic} is a finite formal sum $\Y= \sum_i a_i \gamma_i$, where $a_i\in \Rb_{> 0}$ and $\gamma_i\in \Gamma$ are of infinite order and in distinct conjugacy classes. We define length and norm in the obvious way: $\ell_\rho(\Y) = \sum_ia_i\ell_\rho(\gamma_i)$ and $\norm{\Y} = \sum_i a_i \norm{\gamma_i}$. We let $\widehat\Y$ denote the weighted sum of measures $\sum_i a_i\widehat{\gamma_i}$, or its projectivization.

We define $r(\sum_i a_i\gamma_i) = \min_i r(\gamma_i)$. For convenience we let $\MM(\Gamma)$ denote the set of weighted multigeodesics. 

We now introduce several notions of ending currents (see Section~\ref{section: ending currents} for details). 

\medskip \noindent {\bf Distortion currents.}
The \emph{bounded ending current} of $\rho$ is the collection of limits of counting
measures on closed geodesics whose $\ell_\rho$-lengths are uniformly bounded while their
stable word lengths diverge, that is: 
\[\begin{aligned} \hEbdd(\rho) =  \left\{\lim_{n \to \infty}\widehat{\gamma}_n\in \PCurr(\Gamma): (\gamma_n)\subset\Gamma, \quad
    \begin{array}{l}
    \sup_n\ell_\rho(\gamma_n) < \infty,\\ 
    \lim_{n \to     \infty}\norm{\gamma_n}=+\infty 
    \end{array}\right\}
\end{aligned}~.\]

The \emph{contracting ending current} of $\rho$ is the collection of limits of counting measures on closed geodesics with sublinear $\ell_\rho$-length growth, that is,
\[\hEcont(\rho) = \left\{ \lim_{n \to \infty}\widehat{\gamma}_n\in \PCurr(\Gamma): (\gamma_n)\subset\Gamma,  \quad \lim_{n \to \infty}\frac{\ell_\rho(\gamma_n)}{\norm{\gamma_n}} = 0\right\}~.\] 

\medskip \noindent {\bf Escaping currents.}
The \emph{single escaping ending current} of $\rho$ is the collection of limits of counting measures on geodesics that escape from any compact set in $\rho(\Gamma) \backslash X$:
\[\hEesc(\rho) = \left\{ \lim_{n \to \infty}\widehat{\gamma}_n\in \PCurr(\Gamma): (\gamma_n)\subset\Gamma, \quad r(\gamma_n) \to +\infty\right\}~.\]

The \emph{multiple escaping ending current} of $\rho$ is the collection of limits of weighted multigeodesics that uniformly escape from any compact set in $\rho(\Gamma)\backslash X$:
\[ \hEmesc(\rho) = \left\{ \lim_{n \to \infty}\widehat\Y_n\in \PCurr(\Gamma): (\Y_n)\subset\MM(\Gamma), \quad r(\Y_n) \to +\infty\right\}~.\]

Note that these are all {\em collections of currents} rather than currents. For each such collection $\widehat \Ec$ we denote the corresponding union of supports by
\[\Ec=\bigcup_{\mu\in \widehat{\Ec}}\supp\mu\subset \partial_\infty^{(2)}\Gamma~.\]
(Thus we have $\Econt, \Eesc,$ etc). We will see later that such a union of supports may not be a closed set, and hence is not itself the support of a single current. 

If we replace the above weak-$*$ limits defining $\hEesc$ by Hausdorff limits, we obtain the {\em topologically escaping set}, 
\[\Ehausesc(\rho) = \{ \lim_{n \to \infty}(\gamma_n^+,\gamma_n^-) \in \partial_\infty^{(2)}\Gamma: (\gamma_n)\subset \Gamma, \quad \lim_{n \to \infty} r(\gamma_n) = +\infty\}~.\]

We also recall from \cite{mitra1997ending} the \emph{algebraic ending prelamination} of $\rho$, defined by
\[\begin{aligned}\Lambda(\rho) = \Gamma\cdot \left\{ \lim_{n \to \infty}(\gamma_n^+,\gamma_n^-) \in \partial_\infty^{(2)}\Gamma: \begin{array}{l}(\gamma_n)\subset \Gamma, \text{  an axis of } \gamma_n \text{ passes through } \id, \\ d_X(o,\Axis_{\rho(\gamma_n)}) \to  +\infty \end{array}\right\}~,\end{aligned}\] where $\Axis_{\rho(\gamma_n)}$ denote the (some) axis of $\rho(\gamma_n)$ in $X$.

\medskip \noindent {\bf Cannon--Thurston currents.}
If $\rho$ admits a Cannon--Thurston map $\xi_\rho: \partial_\infty \Gamma \to \partial_\infty X$ (see Definition~\ref{definition: Cannon--Thurston map}), we denote its gluing locus by \[\Omega(\rho) = \{(x,y)\in \partial_\infty^{(2)}\Gamma: \xi_\rho(x) = \xi_\rho(y)\}~.\] In this setting Mj \cite[Lemma~2.2]{mitra1999theorem} proved that $\Lambda(\rho)=\Omega(\rho)$.

\medskip \noindent {\bf Relations among the currents.}

\begin{theoremalpha}\label{main theorem: the relations among ending currents}
    Let $\rho:\Gamma \to \Isom(X)$ be a discrete representation with finite kernel. Then
    \begin{enumerate}
        \item $\hEbdd(\rho) \subset \hEesc(\rho) \subset \hEmesc(\rho) \subset \hEcont(\rho)$;
        \item $\Ehausesc(\rho) \subset \Lambda(\rho)$;
    \end{enumerate}
    If $\rho$ further admits a Cannon--Thurston map, then
    \begin{enumerate}[resume]
        \item $\hEmesc(\rho) = \hEcont(\rho)$;
        \item $\Emesc(\rho) = \Econt(\rho) \subset \Ehausesc(\rho) \subset \Lambda(\rho) = \Omega(\rho)$;
        \item If $\mu\in \PCurr(\Gamma)$ is such that $\supp \mu \subset \Omega(\rho)$, then $\mu \in \hEcont(\rho)$.
    \end{enumerate}
\end{theoremalpha}

In the $3$-dimensional Kleinian surface group setting, these notions recover the ending laminations.

\begin{proposition}\label{proposition: 3 dimensional ending current union of ending laminations}
    Suppose $\Sigma$ is a closed hyperbolic surface and $\rho:\pi_1(\Sigma) \to \Isom^+(\Hb^3)$ is a discrete and faithful representation. Then \[\Ebdd(\rho)=\Eesc(\rho)=\Emesc(\rho)=\Econt(\rho)\] is the union of the ending laminations of the simply degenerate ends and the parabolic curves, regarded as subsets of $\partial_\infty^{(2)}\pi_1(\Sigma)$.
\end{proposition}

In higher dimension, however, these notions of ending current capture genuinely new phenomena. Besides the example constructed in Theorem~\ref{main theorem: the example of simple-stable representations}, where all ending currents consist of a single self-intersecting closed geodesic, we also show that the bounded and contracting ending currents may differ.

\begin{theoremalpha}\label{main theorem: triangles limit to a polygon}
    Suppose $\Sigma$ is a closed hyperbolic surface. Then there exists a Gromov hyperbolic space $X$ and a discrete and faithful representation $\rho: \pi_1(\Sigma) \to \Isom(X)$ such that \[\Ebdd(\rho) \subsetneq \Econt(\rho)~.\]
\end{theoremalpha}

To understand the gap between the topological endings and the measure-theoretic ending currents, we introduce a density result in the space of geodesic currents. The density of counting measures of closed geodesics in the space of projectivized currents is classical, due to Bonahon \cite{bonahon1988geometry,bonahon1991geodesic}. In our setting, we prove a stronger density property that every geodesic current can be approximated by weighted multigeodesics in both weak* topology and Hausdorff distance sense.

\begin{theoremalpha}\label{main theorem: density of multigeodesics}
    For any geodesic current on $\Gamma$, there exists a sequence of weighted multigeodesics that strongly converges to it.
\end{theoremalpha}

To complete the proof, we also establish an Anosov closing lemma in general on the geodesic flow (such as Mineyev's flow) over a hyperbolic group (see Theorem~\ref{theorem: a closing lemma}).

\subsection{Hyperbolic group extensions}\label{introduction section: hyperbolic group extensions}

The study of Cannon--Thurston maps originally arose from Kleinian groups and surface bundles over the circle \cite{cannon-thurston}. In a series of papers, Mj \cite{mitra1997ending,mitra1998cannon,mitra1999theorem} and Mj--Rafi \cite{mj2018algebraic} generalized this phenomenon to abstract hyperbolic group extensions. In particular, \cite{mitra1998cannon} shows that if $\Gamma$ and $G$ are hyperbolic groups and $\Gamma$ embeds in $G$ as a normal subgroup, then the Cannon--Thurston map $\partial_\infty\Gamma\to \partial_\infty G$ exists.

When $\Gamma = \pi_1(\Sigma)$ where $\Sigma$ is a closed hyperbolic surface, these extensions are given by convex-cocompact subgroups of the mapping class group. More precisely, Farb--Mosher \cite{farb2002convex} introduced the notion of convex-cocompact subgroup of the mapping class group. If $H\subset \Mod(\Sigma)$ is convex-cocompact, Farb--Mosher \cite{farb2002convex} and Hamenst\"adt \cite{hamenstadt2005word} showed that the corresponding group extension $G$ given by the short exact sequence \[1 \longrightarrow \pi_1(\Sigma) \stackrel{\rho}{\longrightarrow} G \stackrel{p}{\longrightarrow} H \longrightarrow 1~,\] is hyperbolic. Conversely, up to quotienting by the finite centralizer of $\pi_1(\Sigma)$ and up to taking the orientation-preserving subgroup, a hyperbolic surface-group extension  $G$ of $\pi_1(\Sigma)$ arises from a subgroup of $\Mod(\Sigma)$ in this way. Equivalently, by the work of Kent--Leininger \cite{kent2008shadows} and independently Hamenst\"adt \cite{hamenstadt2005word}, convex-cocompactness is characterized by the condition that an orbit map to the curve complex $H\to \mathcal{C}(\Sigma)$ is a quasi-isometric embedding. The Gromov boundary of the curve complex is identified with the ending lamination space (\cite{klarreich2022boundary}). Hence a convex-cocompact subgroup $H$ comes equipped with an equivariant limit map \[\xi_H:\partial_\infty H\to \PML(\Sigma)~,\] whose image consists of filling and uniquely ergodic laminations.

On the other hand, Mj and Mj--Rafi \cite{mitra1997ending,mj2018algebraic} introduced the algebraic ending prelamination to describe the Cannon--Thurston lamination of the extension. In particular, in the present setting they prove \[\Lambda(\rho) = \Omega(\rho) = \bigcup_{z \in \partial_\infty H} (\supp \xi_H(z))^d~,\] where $(\supp \xi_H(z))^d$ denotes the diagonal closure of $\supp \xi_H(z)$. 

Passing from the topological framework to our measure-theoretic one, we show the following theorem. We regard $\PML(\Sigma)$ as a subset of $\PCurr(\pi_1(\Sigma))$ via the canonical embedding of measured laminations into geodesic currents (\cite{bonahon1988geometry,bonahon1991geodesic}).

\begin{theoremalpha}\label{main theorem: limit set of convex-cocompact representations}
    In the setting above, \[\hEbdd(\rho)=\xi_H(\partial_\infty H) \quad \text{ and hence } \quad \Ebdd(\rho)=\bigcup_{z\in \partial_\infty H}\supp \xi_H(z)~.\]
\end{theoremalpha}

The theorem may be viewed as a measure-theoretic refinement of the algebraic ending lamination picture for hyperbolic group extensions. By Proposition~\ref{proposition: 3 dimensional ending current union of ending laminations}, when $H = \Zb$ is generated by a pseudo-Anosov mapping class $\phi$, the extension $G$ is a mapping torus group and \[\Ebdd(\rho) = \Econt(\rho)\] is the union of the attracting and repelling laminations of $\phi$ (see Example~\ref{example: mapping torus}). In contrast, when $H$ is non-elementary, we construct an example such that \[\Ebdd(\rho) \subsetneq \Econt(\rho)\] (see Example~\ref{example: triangles limit to a polygon}), as a proof of Theorem~\ref{main theorem: triangles limit to a polygon}.

\subsection{Geometric finiteness}

A second application concerns geometric finiteness. We ask to what extent geometric finiteness can be characterized purely in terms of the various ending data introduced above.

For a quasiconvex subgroup $P\subset\Gamma$, we set $[P]:=\Gamma\cdot\partial_\infty^{(2)}P$, the closure of the collection of ordered pairs of endpoints of axes of elements in $P$ or its conjugates. 
For an infinite-order element $\alpha\in\Gamma$, we set $[\alpha]:=[\langle\alpha\rangle]$, then
\[[\alpha] =\Gamma\cdot\{(\alpha^+,\alpha^-),(\alpha^-,\alpha^+)\} = \supp\widehat{\alpha}~.\]

\begin{theoremalpha}\label{main theorem: geometric finiteness and ending currents}
    Suppose $\rho:\Gamma \to \Isom(X)$ is a discrete representation with finite kernel that admits a Cannon--Thurston map $\xi_\rho$ and 
    suppose $\Pc$ is a finite, almost malnormal collection of finitely generated, infinite quasiconvex subgroups of $\Gamma$, with one representative from each conjugacy class. Here almost malnormal means for any $P,Q\in \Pc$, $P\cap gQg^{-1}$ is finite unless $g\in P$ and $P=Q$.
    \begin{enumerate}
        \item\label{item: geometrically finite} $\rho$ is geometrically finite with collection of parabolic subgroups $\Pc$;
        \item\label{item: yaman} $\Omega(\rho) = \bigcup_{P\in \Pc} [P]$;
        \item\label{item: topological ending} $\Ehausesc(\rho) = \bigcup_{P\in \Pc} [P]$;
        \item\label{item: contracting measure} $\Econt(\rho) = \bigcup_{P\in \Pc} [P]$;
        \item\label{item: single far away} $\Eesc(\rho) = \bigcup_{P\in \Pc} [P]$;
        \item\label{item: bounded measure} $\Ebdd(\rho) = \bigcup_{P\in \Pc} [P]$;
    \end{enumerate}
    we have \eqref{item: geometrically finite} $\Leftrightarrow$ \eqref{item: yaman} $\Leftrightarrow$ \eqref{item: topological ending} $\Rightarrow$ \eqref{item: contracting measure} $\Rightarrow$ \eqref{item: single far away} $\Rightarrow$ \eqref{item: bounded measure}.
\end{theoremalpha}

The equivalence among \eqref{item: geometrically finite}, \eqref{item: yaman} and \eqref{item: topological ending} is completely clean. However, \eqref{item: contracting measure}, \eqref{item: single far away} and \eqref{item: bounded measure} involve a subtle ``mass vanishing'' phenomenon, that is, a sequence of escaping loxodromic elements in $\Gamma$ may have their axes spend most of their length deep inside parabolic regions, so that its normalized counting measures converge entirely to parabolic components. We construct such an example.

\begin{theoremalpha}\label{main theorem: geometrically infinite example}
    Let $\Gamma$ be the free group on two generators $\alpha,\beta$. There exists a Gromov hyperbolic space $X$ and a discrete faithful representation $\rho:\Gamma \to \Isom(X)$ such that 
    \begin{enumerate}
        \item $\hEcont(\rho)= \{\widehat \alpha\}$ and hence $\Econt(\rho) = [\alpha]$;
        \item $\rho$ admits a Cannon--Thurston map;
        \item $\Econt(\rho) \subsetneq \Omega(\rho)$.
    \end{enumerate}
\end{theoremalpha}

\subsection{Related topics on mapping class group dynamics}

The example we constructed in Theorem~\ref{main theorem: the example of simple-stable representations} can also be applied in the study of the dynamics of mapping class group action on character varieties. Let $\Gamma$ be a hyperbolic group and let $G$ be a semisimple Lie group of rank $1$. The outer automorphism group $\Out(\Gamma)$ acts by precomposition on the character variety \[\chi(\Gamma,G) = \Hom(\Gamma,G) // G~,\] the GIT quotient of $\Hom(\Gamma,G)$ under the conjugation action of $G$. It is well-known that $\mathcal{CC}(\Gamma,G)$, the set of convex-cocompact characters of $\Gamma$ into $G$, is a domain of discontinuity in $\chi(\Gamma,G)$, that is, an open subset with the $\Out(\Gamma)$-action properly discontinuous.

When $\Gamma$ is a free group, the first author \cite{minsky2013dynamics} introduced the notion of primitive-stable representations and proved that they form a domain of discontinuity strictly containing $\mathcal{CC}(\Gamma,G)$.

Constructing such larger domains of discontinuity is more subtle when $\Gamma$ is a closed hyperbolic surface group. We still assume that $\Sigma$ is a closed oriented surface of genus at least $2$, and view the mapping class group $\Mod(\Sigma)$ as an index $2$ subgroup of $\Out(\pi_1(\Sigma))$, acting on the character variety $\chi(\pi_1(\Sigma),G)$.

When $G = \Isom^+{\Hb^2}$, Goldman \cite{goldman2005mapping} conjectured that the $\Mod(\Sigma)$-actions on the components of $\chi(\pi_1(\Sigma),\Isom^+{\Hb^2})$ except the two Teichm\"uller components are ergodic. When $G = \Isom^+{\Hb^3}$, Canary--Storm \cite{canary2013moduli} and Souto--Storm \cite{SS06} showed that the domain of quasi-Fuchsian representations (i.e., convex-cocompact representations) in $\chi(\pi_1(\Sigma),\Isom^+(\Hb^3))$ is a maximal connected domain of discontinuity (see also \cite[Proposition~17]{Lee11}).

In contrast, in higher dimension there exist non-convex-cocompact representations which still lie in a domain of discontinuity. More precisely, we show that the representation $\rho_0$ from Theorem~\ref{main theorem: the example of simple-stable representations} is not convex-cocompact, but deforms to a Fuchsian representation and lies in a domain of discontinuity, namely, the domain of simple-stable characters (see Definition~\ref{definition: simple-stable}). As a consequence, we obtain the following.

\begin{theoremalpha}\label{main theorem: simple-stable form large domain of discontinuity}
    The set of simple-stable characters $\mathcal{SS}(\pi_1(\Sigma),\Isom^+(\Hb^n))$ is a domain of discontinuity in $\chi(\pi_1(\Sigma),\Isom^+(\Hb^n))$ for the $\Mod(\Sigma)$-action, which contains $\mathcal{CC}(\pi_1(\Sigma),\Isom^+(\Hb^n))$. When $n\geqslant 4$, $\mathcal{SS}(\pi_1(\Sigma),\Isom^+(\Hb^n))$ has a component that strictly contains a component of $\mathcal{CC}(\pi_1(\Sigma),\Isom^+(\Hb^n))$.
\end{theoremalpha}

Lee \cite{Lee11} proved essentially this result for $n=3$ when $\Sigma$ is nonorientable, and our argument is a variation on her ideas. A recent work of Remfort-Aurat \cite{remfort2026simple} gave a similar result in $\mathsf{PU}(2,1)$.



\subsection{Structure of the paper}

In Section~\ref{section: construction of a simple-stable representation}, we recall the simultaneous uniformization theorem for geometrically finite Kleinian groups and basic facts about simple-stable representations. We then construct the higher dimensional surface group example and prove Theorem~\ref{main theorem: simple-stable form large domain of discontinuity}.

In Section~\ref{section: dynamics on geodesic flows}, we study geodesic flows for hyperbolic groups and prove a closing lemma for the flow space (Theorem~\ref{theorem: a closing lemma}). From this we deduce Theorem~\ref{main theorem: density of multigeodesics}.

We discuss the notions of ending currents in Section~\ref{section: ending currents}. In particular, we introduce the various notions of ending current and relate them to Cannon--Thurston maps and Mj's algebraic ending prelaminations in Section~\ref{subsection: notions of ending current}, where we prove Theorem~\ref{main theorem: the relations among ending currents}. We then prove Theorem~\ref{main theorem: limit set of convex-cocompact representations} and Theorem~\ref{main theorem: triangles limit to a polygon} in Section~\ref{subsection: ending currents of hyperbolic group extensions},  Theorem~\ref{main theorem: geometric finiteness and ending currents} in Section~\ref{subsection: ending currents of geometrically finite representations}, and Theorem~\ref{main theorem: geometrically infinite example} in Section~\ref{section: geometrically infinite example}.

In Section~\ref{section: relatively hyperbolic groups}, we recall some preliminaries on relatively hyperbolic groups. In particular in preparation for the proof of Theorem~\ref{main theorem: geometrically infinite example}, we recall a criterion for relative hyperbolicity in Section~\ref{section: 1/6} and a criterion for the existence of Cannon--Thurston maps in Section~\ref{section: mj-pal}.

\subsection*{Acknowledgments}
The second author would like to thank Junzhi Huang and Eduardo Reyes for many interesting discussions.

This paper is partially based upon work supported by the National Science Foundation under Grant No. DMS-2424139, while the authors were in residence at the Simons Laufer Mathematical Sciences Institute in Berkeley, California, during the Spring 2026 semester. The first author was also supported by National Science Foundation Grant No. DMS-2505914.

\subsection*{Tool Use}

A large language model was used to suggest Lemma~\ref{lemma: mystery computation} as a tool for Proposition~\ref{proposition: econt of geometrically finite counterexample}, and to formulate an initial proof. The proof that appears here was simplified, written, and carefully checked by the authors. 

\section{Construction of a simple-stable representation}\label{section: construction of a simple-stable representation}

\subsection{\texorpdfstring{Rigidity of hyperbolic $3$-manifolds}{Rigidity of hyperbolic 3-manifolds}}\label{section: rigidity}
    
Let $\Sigma$ be a closed oriented hyperbolic surface and $\rho:\pi_1(\Sigma)\to \Isom^+(\Hb^3)$  a geometrically finite representation. Let $\Lc_\rho$ be the limit set of $\rho$ in $\partial_\infty \Hb^3$. Then  $N = \rho(\pi_1(\Sigma))\backslash\big(\Hb^3\cup (\partial_\infty \Hb^3 \setminus\Lc_\rho)\big)$ is homeomorphic to $\Sigma\times [0,1] \setminus P$ equivariantly up to homotopy, where $P = P^+ \cup P^-$, $P^+\subset \Sigma\times \{1\}$, $P^-\subset \Sigma\times \{0\}$ are (embedded) multi-curves and no components of $P^+$ and $P^-$ are homotopic in $\Sigma\times[0,1]$. $P$ precisely records the cusp curves of $N$.

The boundary components $(\Sigma\setminus P^+)\times\{1\}$ and $(\Sigma\setminus P^-)\times\{0\}$ of $\Sigma\times [0,1]\setminus P$ are identified with the boundary components of $\partial N$. Each boundary component gives an element $\mu^\pm \in \Teich(\Sigma\setminus P^\pm)$. We call $P^\pm$ together with $\mu^+,\mu^-$ the \emph{ending data} of $\rho$. We also call $\nu^\pm = (P^\pm, \mu^\pm)$ the \emph{ending invariants} of $\rho$.

The following facts, known as a generalized simultaneous uniformization theorem, follow from results of Ahlfors--Bers \cite{ahlfors1960riemann}, Bers \cite{bers2006spaces}, Kra \cite{kra1972spaces}, Marden \cite{marden1974geometry}, Maskit \cite{maskit1970boundaries}, Mostow \cite{mostow1968quasi}, Prasad \cite{prasad1973strong}, Thurston \cite{thurston1982three} and others on the classification of geometrically finite hyperbolic $3$-manifolds.

\begin{theorem}\label{theorem: simultaneous uniformization}
    Suppose $\Sigma$ is a closed oriented hyperbolic surface.
    \begin{enumerate}
        \item For any multi-curves $P^+,P^-\subset \Sigma$ that share no common component and $\mu^\pm\in \Teich(\Sigma\setminus P^\pm)$, there exists a geometrically finite representation $\rho$ with ending data $(P^\pm,\mu^\pm)$.
        \item A geometrically finite representation $\rho: \pi_1(\Sigma) \to \Isom^+(\Hb^3)$ is determined uniquely up to conjugacy by its ending data.
    \end{enumerate}
\end{theorem}

\subsection{Review of simple-stable representations}

We assume that $\Sigma$ is a closed hyperbolic surface. For a representation $\rho:\pi_1(\Sigma)\to \Isom^+(\Hb^n) \cong \SO^+(n,1)$ and a fixed base point $o\in \Hb^n$, the orbit map $\tau_{\rho,o}:\pi_1(\Sigma)\to \Hb^n$, $\gamma \mapsto \rho(\gamma)\cdot o$ reflects many significant geometric properties of $\rho$. We fix a finite generating set $S$ of $\pi_1(\Sigma)$ and let $\abs{\cdot}_S$ denote the word metric. Recall that $\rho$ is \emph{convex-cocompact} if $\tau_{\rho,o}$ is a quasi-isometric embedding. Simple-stability, as an analog of primitive-stability for free groups, requires the quasi-isometric embedding property only along simple closed geodesics in $\Sigma$.

\begin{definition}\label{definition: simple-stable}
    A representation $\rho:\pi_1(\Sigma)\to \SO^+(n,1)$ is simple-stable if there exist constants $C\geqslant 1 $ and $c\geqslant 0$, such that $\tau_{\rho,o}$ maps the axes of every element represented by a simple closed geodesic on $\Sigma$ to a $(C,c)$-quasi-geodesic in $\Hb^n$.
\end{definition}

We denote the set of all simple-stable representations of $\pi_1(\Sigma)$ into $\SO^+(n,1)$ modulo conjugation by $\mathcal{SS}(\pi_1(\Sigma),\SO^+(n,1))$. The mapping class group $\Mod(\Sigma)$ acts on the character variety $\chi(\pi_1(\Sigma),\SO^+(n,1))$ by precomposition and preserves $\mathcal{SS}(\pi_1(\Sigma),\SO^+(n,1))$ as $\Mod(\Sigma)$ preserves the set of simple closed geodesics on $\Sigma$.

\begin{theorem}\label{theorem: simple-stable form domain of discontinuity}
    $\mathcal{SS}(\pi_1(\Sigma),\SO^+(n,1))$ is an open subset in the character variety  $\chi(\pi_1(\Sigma),\SO^+(n,1))$ and the $\Mod(\Sigma)$-action on $\mathcal{SS}(\pi_1(\Sigma),\SO^+(n,1))$ is properly discontinuous.
\end{theorem}

The theorem can be shown by a similar argument as in \cite[Lemma~3.2, Theorem~3.3]{minsky2013dynamics} and \cite[Proposition~10]{Lee11} based on a local-to-global argument, or \cite[Theorem~5.3]{TW23} based on a contracting flow argument.

\subsection{Construction of the example}\label{section: construction of the example}
We now fix an arbitrary closed oriented hyperbolic surface $\Sigma$. Let $\pi:\widetilde{\Sigma}\to \Sigma$ be a double covering of $\Sigma$. Then there exists a simple closed geodesic $\widetilde{c}$ on $\widetilde{\Sigma}$ such that $c = \pi(\widetilde{c})$ on $\Sigma$ has essential self-intersections. We denote by $\tau$ the involution in the deck group $\Gal(\widetilde{\Sigma}/\Sigma)\cong \Zb/2\Zb$. We fix a base point $x_0\in \widetilde{\Sigma}$, let $\alpha\in \pi_1(\Sigma,\pi(x_0))$ be such that it lifts to a path from $x_0$ to $\tau(x_0)$, which induces an isomorphism $\pi_1(\widetilde{\Sigma},x_0)\cong \pi_1(\widetilde{\Sigma},\tau(x_0))$.

We fix some $\mu \in \Teich(\widetilde{\Sigma}\setminus \widetilde{c})$, and let $\nu^+=(\widetilde{c},\mu)$ be an ending invariant on $\widetilde\Sigma$. Let $\nu^- = \tau(\nu^+)$. Then by Theorem~\ref{theorem: simultaneous uniformization}~(1), there exists a geometrically finite representation $\rho: \pi_1(\widetilde{\Sigma},x_0)\to \Isom^+(\Hb^3) \cong \SO^+(3,1)$ with ending invariants $(\nu^+,\nu^-)$ with the parabolic subgroups given by the conjugates of $\ip{\rho(\widetilde{c})}$ and $\ip{\rho(\tau(\widetilde{c}))}$.

On the other hand $\rho\circ \tau:\pi_1(\widetilde{\Sigma},\tau(x_0))\to \SO^+(3,1)$ can be understood as a representation from $\pi_1(\widetilde{\Sigma},x_0)$ by precomposing the isomorphism between $\pi_1(\widetilde{\Sigma},x_0)$ and $\pi_1(\widetilde{\Sigma},\tau(x_0))$ induced by $\alpha$. In particular, we have $\rho\circ\tau(\alpha^{-1}\gamma\alpha) = \rho(\gamma)$ for any $\gamma\in\pi_1(\widetilde{\Sigma},x_0)$.

Notice that $\rho\circ \tau$ has ending invariants $(\nu^-,\nu^+)$.
Suppose $\Phi_0\in \mathsf{O}^+(3,1)$ is an orientation reversing isometry. Then $\Phi_0\rho(\gamma)\Phi_0^{-1}$ also has ending invariants $(\nu^-,\nu^+)$. Then by Theorem~\ref{theorem: simultaneous uniformization}~(2), there exists an orientation preserving isometry $\Phi_1\in \SO^+(3,1)$ such that $\Phi_1\Phi_0\rho(\gamma)\Phi_0^{-1}\Phi_1^{-1} = \rho(\tau(\gamma))$. We write $\Phi = \Phi_1\Phi_0$.

Since $\pi_1(\widetilde{\Sigma},x_0)$ is a subgroup of $\pi_1(\Sigma,\pi(x_0))$ of index $2$ and $\alpha$ is lifted to a path from $x_0$ to $\tau(x_0)$, then $\alpha$ is a representative of the non-identity coset of $\pi_1(\widetilde{\Sigma},x_0)$. By these two relations $\Phi\rho(\gamma)\Phi^{-1} = \rho(\tau(\gamma))$ and $\rho(\tau(\alpha^{-1}\gamma\alpha)) = \rho(\gamma)$ for any $\gamma\in \pi_1(\widetilde{\Sigma},x_0)$, it is immediate to see

\begin{lemma}
    The conditions $\rho_0 \vert_{\pi_1(\widetilde{\Sigma},x_0)} = \rho $ and $\rho_0(\alpha) = \Phi$ uniquely define a representation $\rho_0:\pi_1(\Sigma,\pi(x_0)) \to \mathsf{O}^+(3,1)$.
\end{lemma}

\begin{proof}
    Any element of $\pi_1(\Sigma)$ is either of the form $\gamma$ or $\alpha \gamma$ for some $\gamma\in \pi_1(\widetilde{\Sigma},x_0)$. We define $\rho_0(\gamma) = \rho(\gamma)$ and $\rho_0(\alpha\gamma) = \Phi\rho(\gamma)$. Then $\rho_0$ is defined everywhere on $\pi_1(\Sigma,\pi(x_0))$. To show it is a representation, it suffices to check $\rho_0(\gamma \alpha) = \rho(\gamma) \Phi$ for any $\gamma\in \pi_1(\widetilde{\Sigma},x_0)$, which follows from the computation \[\rho_0(\gamma \alpha) = \rho_0 (\alpha (\alpha^{-1} \gamma \alpha) ) = \Phi \rho(\alpha^{-1} \gamma \alpha) \] \[= (\Phi \rho(\alpha^{-1} \gamma \alpha) \Phi^{-1}) \Phi = \rho(\tau(\alpha^{-1} \gamma \alpha)) \Phi = \rho(\gamma) \Phi~.\]
\end{proof}

$\mathsf{O}^+(3,1)$ is embedded in 
\[\begin{bmatrix} \mathsf{O}^+(3,1) & \\ & \pm 1\end{bmatrix}\subset \SO^+(4,1)\] by adding a $1$ (respectively, -1) at $(5,5)$-entry for each orientation preserving (respectively, reversing) element. We now view $\rho_0$ as a representation from $\pi_1(\Sigma)$ to $\SO^+(4,1)$.

\begin{proposition}
    $\rho_0$ is simple-stable.
\end{proposition}

\begin{proof}[Sketch of the proof]
    In \cite[Proposition 14]{Lee11} Lee proves that, for a discrete representation of a  twisted $I$-bundle group with nonorientable base into $\SO^+(3,1)$, if every simple curve is loxodromic then the representation is simple-stable. In our case we have a twisted $I$-bundle group with orientable base (hence nonorientable $I$-bundle), with representation into $\mathsf{O}^+(3,1)$ (ignoring the action in the normal direction to $\Hb^3$). Lee's proof goes through without changes. Another approach follows \cite{TW23} using an Anosov property.
\end{proof}

\begin{proposition}
    $\rho_0$ deforms continuously to a Fuchsian representation.
\end{proposition}

\begin{proof}
    Since $\rho$ is a limit of quasi-Fuchsian representations, it deforms to a representation from $\pi_1(\widetilde{\Sigma})$ to $\SO^+(2,1)\subset \SO^+(3,1)$. Notice that one can choose $\Phi\in \mathsf{O}^+(3,1)$ that deforms continuously as $\rho$ deforms. Therefore, $\rho_0$ deforms to a representation $\rho_1$ with image in \[\SO^+(2, 1)\times \Zb/2\Zb \subset \SO^+(2, 1)\times \SO(2)\subset \SO^+(4,1)~.\] We claim that $\rho_1$ deforms to a representation $\rho_2$ with image in $\SO^+(2, 1)$. Actually, it suffices to show that a representation $j: \pi_1(\Sigma) \to \SO(2)$ deforms to the trivial representation. Since $\SO(2)\cong \Sb^1$ is abelian, $j$ factors through $\pi_1(\Sigma) \to H_1(\Sigma,\Zb) \to \SO(2)$ and since $\pi_1(\Sigma)$ is torsion free, $H_1(\Sigma,\Zb)$ is a free abelian group. Therefore $j$ deforms continuously to the trivial representation.
\end{proof}

Together with Theorem~\ref{theorem: simple-stable form domain of discontinuity}, this completes the proof of Theorem~\ref{main theorem: simple-stable form large domain of discontinuity}.

\section{Dynamics on geodesic flows}\label{section: dynamics on geodesic flows}

In this section we fix a hyperbolic group $\Gamma$. It is well known that the set of counting measures is dense in the space of projectivized geodesic currents on $\Gamma$ (\cite{bonahon1988geometry}). The goal of this section is to provide a stronger density theorem that any geodesic current on $\Gamma$ is the limit of a sequence of weighted multigeodesics (see Definition~\ref{definition: multigeodesics}), and moreover the support of the current is the Hausdorff limit of the sequence of underlying multigeodesics. We refer to this type of convergence as strong convergence (see Definition~\ref{definition: strong convergence}). Our goal is to prove Theorem~\ref{main theorem: density of multigeodesics}, which we restate here:

\begin{theorem}\label{theorem: density of multigeodesics}
    For any $\mu \in \Curr(\Gamma)$, there exists a sequence of weighted multigeodesics $(\Y_n)_{n\in \Nb}$ that strongly converges to $\mu$.
\end{theorem}

We prove the theorem by fixing a global transversal for the geodesic flow and establishing a closing lemma. We then use partitions of the transversal as coarse analogues of train tracks, to approximate any given current.

The existence of global transversals is standard, but we have not found a reference giving the precise statement needed in Section~\ref{section: A global transversal}. We therefore provide an explicit construction adapted to our setting. The closing lemma is also well known, and we include a proof in our language in Section~\ref{section: Closing lemma}. The conditions we impose on the partitions are weaker than the Markov conditions used in symbolic dynamics, so we need to carefully refine the partitions in the proof of Theorem~\ref{theorem: density of multigeodesics} (see Section~\ref{section: Density of multigeodesics}).

\subsection{Geodesic flows}

For some fixed generating set of $\Gamma$, we denote by $d_\Gamma$ the word metric, $\abs{\gamma}=d_\Gamma(\id,\gamma)$, and $\norm{\gamma}$ the stable length of $\gamma$, that is, $\norm{\gamma} = \lim\limits_{n\to\infty} \frac{\abs{\gamma^n}}{n}$.

Suppose $(X,d_X)$ is a $\delta$-hyperbolic space and $\Gamma$ acts properly discontinuously on $X$ by isometries. Up to quotienting a finite normal subgroup, we may assume that no element $\gamma\in \Gamma\setminus \{\id\}$ preserves $\partial_\infty X$ pointwise.


We denote by $\partial_\infty^{(2)} X$ the set of ordered pairs of distinct points in $\partial_\infty X$ and by $\iota$ the orientation reversing map, that is, \[\begin{matrix} \iota: & \partial_\infty^{(2)} X& \to & \partial_\infty^{(2)} X \\ & (x,y) & \mapsto & (y,x) \end{matrix}~.\] We consider the topological space $\Gc(X) = \partial_\infty^{(2)} X \times \Rb$ with the flow \[\begin{matrix} \phi: & \Gc(X) \times \Rb & \to & \Gc(X) \\ & \big((x,y,s),\ t\big) & \mapsto & \phi^t(x,y,s) = (x,y,s+t) \end{matrix}~.\] 

\begin{definition}\label{definition: the geodesic flow}
    We call $\Gc(X)$ the unit tangent bundle of $X$ if there exists a metric $d$ on $\Gc(X)$ compatible with the product topology, and a continuous map \[c:  \partial_\infty^{(2)} X \times \Gamma  \to  \Rb~,\] that satisfies the symmetric cocycle conditions \[c(x,y,\gamma) = -c(y,x,\gamma) \quad \text{ and } \quad c(x, y, \eta\gamma) = c(\gamma x, \gamma y, \eta)+ c(x,y,\gamma)~,\] for any $(x,y)\in  \partial_\infty^{(2)}X$ and  $\gamma,\eta\in \Gamma$ such that 
    \begin{enumerate}
        \item The $\Gamma$-action on $\Gc(X)$ defined by \[\gamma\cdot(x,y,s) =(\gamma x,\gamma y,s+c(x,y,\gamma))\] is properly discontinuous, $d$-isometric and free, hence the $\Gamma$-action commutes with $\phi$.
        \item There exist constants $A_1\geqslant 1$ and $B_1\geqslant 0$, such that each flow line \[\begin{matrix} (x,y,\cdot): & \Rb & \to & \Gc(X) \\ & t & \mapsto & (x,y,t) \end{matrix}~.\] is an $(A_1,B_1)$-quasi-geodesic;
        \item The involution $\iota$ extends to a $d$-isometry on $\Gc(X)$ by $\iota(x,y,t) = (y,x,-t)$ for any $(x,y)\in\partial_\infty^{(2)}X$, and $\iota$ commutes with the $\Gamma$-action;
        \item There exist constants $A_2\geqslant 1$, $B_2\geqslant 0$, and a $\Gamma$-equivariant $(A_2,B_2)$-quasi-isometry $p_M: \Gc(X)\to X$ that maps each flow line $(x,y,\Rb)$ to a geodesic with forward endpoint $x$ and backward endpoint $y$.
    \end{enumerate}
\end{definition}

By abusing the notation, we still denote the induced flow on $\wh \Gc(X) = \Gc(X) / \Gamma$ by $\phi$.

When $(X,d_X)$ is a uniquely geodesic $\delta$-hyperbolic space (e.g. a hyperbolic manifold), it admits a unit tangent bundle $(\Gc(X),d,\phi)$ given by \[\Gc(X) = \{\ell\ |\ \ell:\Rb \to X\text{ is a geodesic parametrized by length}\}~,\] equipped with the natural $\Gamma$-action and the flow $\phi$ defined by $\phi^t(\ell): s \mapsto \ell(s + t)$. We endow $\Gc(X)$ with a metric $d$ defined by \[d(\ell,\ell') = \int_{\Rb} \frac{e^{-\abs{s}}}{2}d_X\big(\ell(s),\ell'(s)\big)ds~,\] and a projection $p_M:\Gc(X)\to X$ defined by $p_M(\ell)=\ell(0)$. It is straightforward to verify that these data define a unit tangent bundle of $X$. In particular for $\Hb^n$, $\Gc(\Hb^n)$ can be taken as the Hopf parametrized unit tangent bundle of $\Hb^n$. Such unit tangent bundle is independent of $\Gamma$.

When $X$ is a locally finite hyperbolic graph, a unit tangent bundle of $X$ is constructed by Mineyev \cite{Mineyev}. In particular when $X= \Cay(\Gamma)$ is the Cayley graph of $\Gamma$ with respect to some fixed generating set, the construction is also due to Gromov \cite{Gromov}, Mathéus \cite{Matheus} and Champetier \cite{Cham}.

From now on for any given hyperbolic group $\Gamma$, we fix a choice of its  unit tangent bundle $\Gc(\Gamma) = \Gc(\Cay(\Gamma))$.

By direct computation (see for instance \cite[Lemma~3.8]{W23a}) we have that 

\begin{lemma}\label{lemma: flow time quasi isometric to group action}
    For any compact set $F\subset \Gc(\Gamma)$, there exist constants $A_3\geqslant 1$ and $B_3\geqslant 0$ such that for any $v\in F$, if $t\in \Rb$ and $\gamma\in \Gamma$ satisfy $\gamma^{-1} \phi^t(v)\in F$, then \[A_3^{-1}\abs{t}-B_3\leqslant \abs{\gamma} \leqslant A_3\abs{t}+B_3~.\] 
\end{lemma}

\subsection{The visual metric on the boundary}\label{section: visual metric}

Suppose $X$ is a $\delta$-hyperbolic space. Let $o\in X$ be a fixed base point. For any two points $x_0,y_0\in  X$, we denote by $(x_0, y_0)_o$ the Gromov product of $x_0$ and $y_0$ with base point $o$.

The Gromov product of two points $x,y \in \partial_{\infty} X$ with base point $o$ is defined by \[(x, y)_o = \sup \liminf\limits_{m,n \to \infty} (x_m,y_n)_o \] where the supremum is taken from all sequences $\{x_n\}$ and $\{y_n\}$ in $X$ with $x_n\to x$ and $y_n\to y$ as $n \to \infty$. The Gromov product is coarsely well-defined by the following Lemma.

\begin{lemma}[{\cite[Part~III~Remarks~3.17~(5)]{bridson2013metric}}]\label{lemma: Gromov product coarsely well-defined}
    Let $x,y\in\partial_\infty X$, $\{x_n\}$ and $\{y_n\}$ be two sequences in $X$ with $x_n\to x$ and $y_n\to y$ as $n \to \infty$. Then \[(x, y)_o-2\delta\leqslant \liminf\limits_{n \to \infty} (x_n,y_n)_o \leqslant (x, y)_o~.\]
\end{lemma}

Given a constant $\kappa>0$, a pseudo-metric on $\partial_\infty X$ is defined by \[d_\kappa (x,y)= \inf \{ \sum_{i=1}^n e^{-\kappa (x_{i-1}, x_i)_o} : \{x_i\}\subset \partial_\infty X \text{ with } x=x_0, y=x_n \}\] for any $x,y\in\partial_\infty X$. By \cite[Part~III~Proposition~3.21]{bridson2013metric}, $d_\kappa$ is a metric whenever $0< \kappa \leqslant \frac{1}{4\delta} \log 2$ and \[(3-2e^{2\delta\kappa}) e^{-\kappa(x, y)_o} \leqslant d_\kappa (x,y)\leqslant e^{-\kappa (x, y)_o}~.\]

\begin{notation}\label{notation: the visual metric}
    Applying the discussion above to $\Gamma$, we fix such a visual metric $d_\infty$ on $\partial_\infty\Gamma$ given by such a constant $\kappa$, and $0< C_0\leqslant 1$, such that $\forall x,y\in\partial_\infty\Gamma$, \[C_0 e^{-\kappa(x, y)_{\id}} \leqslant d_\infty (x,y)\leqslant e^{-\kappa (x, y)_{\id}}~.\]
\end{notation}

\subsection{A global transversal}\label{section: A global transversal}

In this section, we construct global transversals of $(\Gc(\Gamma),d,\phi)$. Roughly speaking, a global transversal is a subset $K \subset \Gc(\Gamma)$ with compact closure such that the flow starting from any unit tangent vector would hit the $\Gamma$-orbit of $K$ within some uniformly bounded time $T_1$, and a flow line spends at least a uniform time $T_0>0$ between two hits.

\begin{lemma}\label{lemma: existence global transversal}
    There exists a precompact set $K \subset \Gc(\Gamma)$ such that the following hold.
    \begin{enumerate}
        \item $K$ is a union of finitely many rectangle pieces, that is, \[K = \bigsqcup_{\alpha \in \Ac_0} (U_\alpha^+\times U_\alpha^-\times\{t_\alpha\})\] where the index set $\Ac_0$ is finite and each $U_\alpha^\pm$ is an open subset of $\partial_\infty\Gamma$.
        \item There exists $T_0>0$ such that whenever $\phi^t(v), \phi^{t'}(v)\in \Gamma K$ for some $v\in \Gc(\Gamma)$ and distinct $t,t'\in \Rb$, one has $\abs{t-t'}\geqslant T_0$;
        \item There exists $T_1 \geqslant T_0$ such that for any $v\in\Gc(\Gamma)$, there exists $t\in [T_0,T_1]$ such that $\phi^t(v)\in \Gamma K$;
        \item For any $\gamma\in \Gamma$ other than $\id$, $K \cap \gamma K = \emptyset$.
    \end{enumerate}
\end{lemma}

\begin{proof}
    As $\Gamma$ acts on $\Gc(\Gamma)$ properly discontinuously and freely, for any two open subsets $U^+$ and $U^-$ of $\partial_\infty \Gamma$ that are small enough and such that $\overline{U^+}\cap \overline{U^-} = \emptyset$, there exists $\delta>0$ such that \[(U^+\times U^-\times [-\delta,\delta])\cap \gamma(U^+\times U^-\times [-\delta,\delta]) = \emptyset~,\] for any $\gamma\ne \id$. We consider the collection \[\Uc= \{U^+\times U^- \subset \partial_\infty^{(2)}\Gamma : U^+,U^-\text{ are open subsets of } \partial_\infty^{(2)}\Gamma, \] \[\overline{U^+}\cap \overline{U^-} = \emptyset,\text{ and } \exists\ \delta = \delta(U^\pm) \text{ such that }\forall\ \gamma\ne \id, \] \[(U^+\times U^-\times [-\delta,\delta])\cap \gamma(U^+\times U^-\times [-\delta,\delta]) = \emptyset \}~.\]

    Since $\Gamma$ acts on $\Gc(\Gamma)$ cocompactly, let $F\subset \Gc(\Gamma)$ be a compact subset such that $\Gamma F =\Gc(\Gamma)$ and $\iota(F) = F$. Let \[T_M = 1+ \sup \{t\in \Rb: \exists\ (x,y)\in \partial_\infty^{(2)}\Gamma \text{ such that } (x,y,t)\in F\}~.\] Then there exists a finite subset of $\Uc$ denoted by $\{U_i^+\times U_{i}^-\}_{i\in I}$, such that for any $(x,y,t)\in F$, there is some $i\in I$ such that $(x,y)\in U_i^+\times U_i^-$. Let $\delta_0 = \min_{i\in I}\{\delta(U_i^\pm)\}$ and fix some $t_i\geqslant T_M$ for each $i\in I$. Since the $\Gamma$-action and $\phi$ commute on $\Gc(\Gamma)$, then for each $i\in I$, \[(U_i^+\times U_i^-\times [ t_i-\delta_0,t_i+\delta_0])\cap \gamma(U_i^+\times U_i^-\times [ t_i-\delta_0,t_i+\delta_0]) = \emptyset\] for any $\gamma\ne \id$. The closure of \[K_0 = \bigcup_{i\in I} (U_i^+\times U_i^-\times [ t_i-\delta_0,t_i+\delta_0])\] is compact, so we can pick a finite subset $\Bc \subset \Gamma$ such that $K_0\cap \gamma K_0 = \emptyset$ for any $\gamma\in \Gamma\setminus \Bc $. We may also assume $\Bc ^{-1}=\Bc $ and $\id \in \Bc $. Our goal now is to refine the cover $\{U_i^+\times U_{i}^-\}_{i\in I}$. We write $I = \{1,2,\dots,n\}$ as it is finite.\\
    
    \noindent {\bf Step 1. A Refinement of $U_1^+\times U_1^-$.} Let $b = \#\Bc$.
    Let $\delta_1 = \frac{\delta_0}{4nb}$. Recall from Definition~\ref{definition: the geodesic flow} that $\Gamma$ acts on $\Gc(\Gamma)$ via a continuous cocycle $c$. We consider \[\Uc_1 = \{U^+\times U^- : U^\pm\subset U_1^\pm \text{ are open subsets},\ \abs{c(x,y,\gamma)-c(x',y',\gamma)}\leqslant \delta_1\] \[\forall (x,y),(x',y')\in U^+\times U^- \text{ and }\gamma\in \Bc \}~.\] For each $U_\alpha^+\times U_\alpha^- \in \Uc_1$, whenever given $t_\alpha'\in [t_1-\delta_0,t_1+\delta_0]$, $j\in I\setminus\{1\}$, $\gamma\in \Bc$ and $t_j'\in [t_j-\delta_1,t_j+\delta_1]$, a necessary condition for \[\gamma (U_\alpha^+\times U_\alpha^-\times\{t_\alpha'\}) \cap (U_j^+\times U_j^-\times\{t_j'\}) \ne \emptyset\] is that there exists $(x,y)\in U_\alpha^+\times U_\alpha^-$ such that \[c(x,y,\gamma)+t_\alpha' =t_j'~.\] Let $z_\alpha\in U_\alpha^+\times U_\alpha^-$ be a fixed point. Then by comparing the length of intervals, we can find a nonempty open interval \[T_\alpha=(t_\alpha^-,t_\alpha^+)\subset [t_1-\delta_0,t_1+\delta_0]\setminus \bigcup_{j=2}^n \bigcup_{\gamma\in \Bc}[t_j - c(z_\alpha,\gamma)- 2\delta_1,t_j - c(z_\alpha,\gamma) + 2\delta_1]~.\] We conclude that for any $U_\alpha^+\times U_\alpha^-\in \Uc_1$, \[(U_\alpha^+\times U_\alpha^-\times{T_\alpha}) \cap \gamma(U_j^+\times U_j^-\times [t_j-\delta_1,t_j+\delta_1]) = \emptyset\] for any $\gamma\in \Bc $, $j\in I\setminus\{1\}$.\\

    \noindent {\bf Step 2. Refinements of the rest by induction.}
    By the same argument for any $i=2,3,\dots,n$, we take $\delta_{i} = \frac{\delta_{i-1}}{4nb}$ and \[\Uc_i = \{U^+\times U^- : U^\pm\subset U_i^\pm \text{ are open subsets},\ \abs{c(x,y,\gamma)-c(x',y',\gamma)}\leqslant \delta_{i-1}\] \[\forall (x,y),(x',y')\in U^+\times U^- \text{ and }\gamma\in \Bc \}~,\] then find for each $U_\beta^+\times U_\beta^-\in \Uc_i$ an associated nonempty open interval $T_\beta \subset [t_i-\delta_{i-1},t_i+\delta_{i-1}]$, such that \[(U_\beta^+\times U_\beta^-\times{T_\beta}) \cap \gamma(U_j^+\times U_j^-\times [t_j-\delta_i,t_j+\delta_i]) = \emptyset\] for any $\gamma\in \Bc $, $j\in I\setminus\{1,2,\cdots,i\}$.
    
    Finally let $\{U_\alpha^+\times U_\alpha^-\}_{\alpha\in \Ac_0}$ be finite subset of $\bigcup_{i=1}^n \Uc_i$ such that \[\forall v\in F, \ (v^+,v^-) \in \bigcup_{\alpha \in \Ac_0}U_\alpha^+\times U_\alpha^-~.\] We take $t_\alpha\in T_\alpha$ for each $\alpha\in \Ac_0$ such that $t_\alpha \ne t_\beta$ for any $\alpha\ne \beta\in \Ac_0$. Let \[T_0 = \frac{1}{2}\min\big(\{t_\alpha-t_\alpha^-: \alpha \in \Ac_0 \}\cup \{t_\alpha^+ - t_\alpha: \alpha \in \Ac_0\}\big)\]  \[\text{ and }T_1 = \max\{t_\alpha : \alpha\in \Ac_0\} + 2T_M~.\] Then it is easy to verify  the set \[K=\bigsqcup_{\alpha\in \Ac_0}(U_\alpha^+\times U_\alpha^- \times \{t_\alpha\})\] satisfies the required conditions.
\end{proof}

We call such a $K$ constructed in Lemma~\ref{lemma: existence global transversal} a \emph{global transversal} of the flow $(\Gc(\Gamma),d,\phi)$. We denote by $\wh K$ the image of $K$ via the projection $\Gc(\Gamma) \to \wh \Gc(\Gamma)$. Notice that by Lemma~\ref{lemma: existence global transversal}~(4), $K$ is homeomorphic to $\wh K$.

For a compact set $F\subset \Gc(\Gamma)$, we call \[F^\epsilon = \bigcup_{(x,y,t)\in F} \Nc_\epsilon(x)\times \Nc_{\epsilon}(y) \times \{t\} \] the \emph{$\epsilon$-thickening} of $F$.

\begin{remark}\label{remark: thicken a global transversal}
     Suppose we are given a global transversal $K$. There exists a sufficiently small constant $\epsilon(K)>0$, such that for any $0< \epsilon \leqslant \epsilon(K)$, $K^\epsilon$ is again a global transversal, up to shrinking the constant $T_0$.
\end{remark}

\begin{remark}\label{remark: reverse flow transversal}
    Suppose \[K = \bigsqcup_{\alpha \in \Ac} (U_\alpha^+\times U_\alpha^-\times\{t_\alpha\}) \subset \Gc(\Gamma)\] is a global transversal for $(\Gc(\Gamma),d,\phi)$ and $T_0,T_1$ are the constants given by Lemma~\ref{lemma: existence global transversal}. Then \[\iota (K) = \bigsqcup_{\alpha \in \Ac} (U_\alpha^-\times U_\alpha^+\times\{-t_\alpha\}) \subset \Gc(\Gamma)\] is a global transversal for $(\Gc(\Gamma),\phi^{-t})$ with the conditions in Lemma~\ref{lemma: existence global transversal} hold by the same constants $T_0$ and $T_1$.
\end{remark}

\begin{lemma}[Coarsely Anosov flow]\label{lemma: the geodesic flow is coarsely Anosov}
    For any compact subset $F\subset \Gc(\Gamma)$ and $\epsilon>0$ such that $F^\epsilon$, the $\epsilon$-thickening of $F$, is again compact, there exists $A_0\geqslant 1$ such that whenever $v=(x,y,s)\in F$, $t>0$ and $\gamma\in \Gamma$ are such that $\phi^t(v)\in \gamma F$, then for any $y'\in \partial_{\infty}\Gamma$ such that $d_{\infty}(y,y')\leqslant \epsilon$, \[A_0^{-1} e^{-\kappa\abs{\gamma}} d_{\infty}(y,y')\leqslant d_{\infty}(\gamma^{-1}y,\gamma^{-1}y')\leqslant A_0 e^{-\kappa\abs{\gamma}} d_{\infty}(y,y')~;\]
\end{lemma}

\begin{proof}
    Let $\delta>0$ be such that $\Gamma$ is $\delta$--hyperbolic and let $C_0,\kappa>0$ be the constants associated to the visual metric given by Notation~\ref{notation: the visual metric}. Since $d_\infty(y,y') \leqslant \epsilon$, $(x,y',s)\in F^\epsilon$.
    Let $F_0$ be a compact subset of $\Gc(\Gamma)$ that contains both $\id$ and $F^\epsilon$. Both geodesic rays $(x,y,[s,+\infty))$ and $(x,y',[s,+\infty))$ have one endpoint in $F^\epsilon$ and the other endpoint $x$. Then by the Morse lemma of fellow traveling geodesics, up to replacing $F_0$ by a larger compact set, both geodesics $(x,y,s+t))$ and $(x,y',s+t))$ pass through $\gamma F_0$. Let $(y_n),(y_n')\subset\Gamma$ be sequences along these two geodesics converging to $y,y'$ respectively. Then \[(y_n,y_n')_\gamma-(y_n,y_n')_{\id}=\frac{1}{2}(d_\Gamma(y_n,\gamma)+d_\Gamma(y_n',\gamma)-d_\Gamma(y_n,\id)-d_\Gamma(y_n',\id)),\] \[\abs{d_\Gamma(y_n,\gamma)-d_\Gamma(y_n,\id)-\abs{\gamma}}\leqslant 4\diam(F_0)~,\] \[\text{ and } \quad \abs{d_\Gamma(y_n',\gamma)-d_\Gamma(y_n',\id)-\abs{\gamma}} \leqslant 4\diam(F_0)~.\] Thus \[\abs{(y_n,y_n')_\gamma-(y_n,y_n')_{\id}-\abs{\gamma}}\leqslant 8\diam(F_0)~.\] Passing to the limit and applying Lemma~\ref{lemma: Gromov product coarsely well-defined} yields \[\abs{(y,y')_\gamma-(y,y')_{\id}-\abs{\gamma}}\leqslant 8\diam(F_0)+4\delta~.\] 
    Also by Lemma~\ref{lemma: Gromov product coarsely well-defined} we have that 
    \[\abs{(\gamma^{-1}y,\gamma^{-1}y')_{\id} - (y,y')_\gamma}\leqslant 4\delta~.\] Let $A_0 = \frac{1}{C_0} e^{\kappa(8\diam(F_0)+8\delta)}$, then by Notation~\ref{notation: the visual metric} we get the visual metric bounds \[d_\infty(\gamma^{-1}y,\gamma^{-1}y') \leqslant  e^{-\kappa(\gamma^{-1}y,\gamma^{-1}y')_{\id}} \leqslant  e^{-\kappa(y,y')_{\id}+ 8\delta + 8\diam(F_0)} \leqslant A_0 e^{-\kappa\abs{\gamma}}d_\infty(y,y')~,\] \[ \text{ and similarly, }\quad d_\infty(\gamma^{-1}y,\gamma^{-1}y') \geqslant A_0^{-1} e^{-\kappa\abs{\gamma}} d_\infty(y,y')~.\] These complete the proof.
\end{proof}

\begin{notation}\label{notation: the choice of constants before closing lemma}
    From now on, we fix a global transversal $K$. Let $A_3,B_3$ be the constants given by Lemma~\ref{lemma: flow time quasi isometric to group action} for $F = K$, let $\epsilon(K)>0$ be the constant given by Remark~\ref{remark: thicken a global transversal}, and let $A_0\geqslant 1$ be the constant given by Lemma~\ref{lemma: the geodesic flow is coarsely Anosov} for $F = K$ and $\epsilon = \epsilon(K)$. Since $K^{\epsilon(K)}$, the $\epsilon(K)$-thickening of $K$, is again a global transversal by Remark~\ref{remark: thicken a global transversal}, let $0<T_0\leqslant T_1$ be the constants given by Lemma~\ref{lemma: existence global transversal} for $K^{\epsilon(K)}$. Let $L_2>0$ be such that \[5Ae^{-\kappa L_2} \leqslant \frac{1}{2}~,\] and $T_2 = A_3 (L_2 +B_3)$.
\end{notation}

\begin{definition}[Partition of global transversals]
    Let $\epsilon\leqslant \epsilon(K)$. We say a decomposition of $K$ into disjoint subsets \[K = \bigsqcup_{\alpha \in \Ac} K_{\alpha} \subset \Gc(\Gamma)\] is an \emph{$\epsilon$-partition} of $K$ if 
    \begin{enumerate}
        \item $\Ac$ is finite;
        \item For each $\alpha\in \Ac$, $K_\alpha$ is of the form \[K_{\alpha} = U_{\alpha}^+\times U_{\alpha}^-\times\{t_{\alpha}\}~.\]
        Here $U_{\alpha}^\pm$ are not necessarily open.
        \item The $d_\infty$-diameters of $U_{\alpha}^+$ and $U_{\alpha}^-$, denoted by $\diam_\infty(U_{\alpha}^+)$ and $\diam_\infty(U_{\alpha}^-)$ respectively, are both bounded by $\epsilon$ for each $\alpha\in\Ac$.
    \end{enumerate}
    We denote by $\wh K_\alpha$ the image of $K_\alpha$ via the projection $\Gc(\Gamma) \to \wh \Gc(\Gamma)$ and denote \[\Nc_{\epsilon}(K_\alpha) = \Nc_{\epsilon}(U_\alpha^+)\times \Nc_{\epsilon}(U_\alpha^-)\times\{t_\alpha\}~.\]
\end{definition}

We will sometimes refer to a partition \[K = \bigsqcup_{\alpha\in \Ac} K_{\alpha} \subset \Gc(\Gamma)~,\] simply by its index set, writing $\mathcal A$ representing the partition $\{K_\alpha\}_{\alpha \in \mathcal A}$.

\subsection{Closing lemma}\label{section: Closing lemma}

\begin{theorem}[A closing lemma]\label{theorem: a closing lemma}
    For any $R_{\Haus} > 0$ and $T_3 \geqslant T_2$, there exists a constant $\epsilon_0>0$ such that the following statement holds. Suppose we are given $\Ac$, an $\epsilon_0$ partition of $K$, $(\alpha_i)_{i=0}^n\subset \Ac$ with $\alpha_0 = \alpha_n$, $\wh v_i \in \wh K_{\alpha_i}$, and $t_i >0$ such that $\phi^{t_i}(\wh v_i) \in \wh K_{\alpha_{i+1}}$ and $T_2\leqslant t_i \leqslant T_3$ for each $0\leqslant i \leqslant n-1$. Then there exists a loxodromic element $\gamma \in \Gamma$ such that
    the Hausdorff distance between $\bigcup_{i=0}^{n-1} \phi^{[0,t_i]}(\wh v_i)$ and the axis of $\gamma$ in $\wh \Gc(\Gamma)$ is bounded by $R_{\Haus}$.
\end{theorem}

\begin{proof}
    For any $T_3 > 0$ and $R_{\Haus}>0$, we take $\epsilon_0\leqslant \frac{1}{5}\epsilon(K)$ to be a constant such that if $(x,y,s),(x',y',s)\in K^{\epsilon(K)}$ are such that $d_\infty(x,x')\leqslant 5\epsilon_0$ and $d_\infty(y,y')\leqslant 5\epsilon_0$, then for any $t\leqslant T_3$, the Hausdorff distance between $\phi^{[0,t]}(x,y,s)$ and $\phi^{[0,t]}(x',y',s)$ is bounded by $R_{\Haus}$. Roughly speaking, $T_2$ is chosen so that, after flowing for time at least $T_2$, the visual distance contracts sufficiently by Lemma~\ref{lemma: the geodesic flow is coarsely Anosov}. On the other hand, $T_3$is an arbitrary upper bound on the time interval under the consideration that once the initial points are sufficiently close in the boundary directions, the corresponding orbit segments of length at most $T_3$ remain within uniformly bounded Hausdorff distance.

    For each $0\leqslant i \leqslant n-1$, let $v_i$ be the unique preimage of $\wh v_i$ in $K_{\alpha_i}$ of $\overline{v_i}$ via the projection $\Gc(\Gamma) \to \wh \Gc(\Gamma)$. Then there exists $(\eta_i)_{i=1}^n \subset \Gamma$ such that $\phi^{t_i}(v_i) \in \eta_{i+1}K_{\alpha_{i+1}}$. Set $\gamma_0 = \id$ and $\gamma_i = \eta_1 \eta_2 \dots \eta_i$ for each $1\leqslant i \leqslant n$. (See Figure~\ref{figure: closing lemma}.)
    \begin{figure}[ht]
        \include{Figure_closinglemma.tex}
        \caption{}
        \label{figure: closing lemma}
    \end{figure}
    We write $v_i=(x_i,y_i,t_{\alpha_i})\in K_{\alpha_{i}}$, then $\phi^{t_i}(v_i)= (x_i,y_i,t_{\alpha_i}+t_i) \in \eta_{i+1}K_{\alpha_{i+1}}$. We also write $(x_i',y_i',t_{\alpha_{i+1}}) = \eta_{i+1}^{-1} (x_i,y_i,t_{\alpha_i} + t_i) = \eta_{i+1}^{-1} \phi^{t_i}(v_i) \in K_{\alpha_{i+1}}$.

    By our assumption each $K_{\alpha_i}$ satisfies \[\diam_\infty(U_{\alpha_i}^\pm)\le\epsilon_0 \quad \text{ and hence }\quad \diam_\infty(\Nc_{2\epsilon_0}(U_{\alpha_i}^\pm))\le 5\epsilon_0~.\] So we have $d_\infty(x_i',x_{i+1}) \leqslant \epsilon_0$ and $d_\infty(y_i',y_{i+1}) \leqslant \epsilon_0$. 

    Since $t_i\geqslant T_2$, then by Lemma~\ref{lemma: flow time quasi isometric to group action}, $\abs{\eta_i}\geqslant L_2$. By Lemma~\ref{lemma: the geodesic flow is coarsely Anosov}, \[\diam_\infty(\eta_{i+1}^{-1} \Nc_{2\epsilon_0}(U_{\alpha_i}^-))\leqslant 2A e^{-\kappa \abs{\eta_{i+1}}}\diam_{\infty}(\Nc_{2\epsilon_0}(U_{\alpha_i}^-))\leqslant 2Ae^{-\kappa L_2} 5\epsilon_0\leqslant \epsilon_0~.\] Since moreover \[y_i' \in (\eta_{i+1}^{-1}\Nc_{2\epsilon_0}(U_{\alpha_i}^-))\cap U_{\alpha_{i+1}}^- \ne \emptyset~.\] Then we have \[\eta_{i+1}^{-1}\Nc_{2\epsilon_0}(U_{\alpha_i}^-) \subset \Nc_{2\epsilon_0}(U_{\alpha_{i+1}}^-)~.\] We write $\eta_k = \eta_i$ whenever $k$ is congruent to $i$ modulo $n$. Then \[\eta_{i+j}^{-1}\dots \eta_{i+1}^{-1}  \Nc_{2\epsilon_0}(U_{\alpha_{i}}^-) \subset \Nc_{2\epsilon_0}(U_{\alpha_{i+j}}^-)\] Then by Lemma~\ref{lemma: the geodesic flow is coarsely Anosov} \[\diam_\infty(\eta_{i+j}^{-1}\dots \eta_{i+1}^{-1}  \Nc_{2\epsilon_0}(U_{\alpha_{i}}^-))\leqslant 2 Ae^{-\kappa j L_2}  \diam_\infty\Nc_{2\epsilon_0}(U_{\alpha_{i}}^-)\leqslant \frac{1}{10^{j-1}}\epsilon_0~.\]

    In particular, \[\Nc_{2\epsilon_0}(U_{\alpha_0}^-) \supset \gamma_{n}^{-1} \Nc_{2\epsilon_0}(U_{\alpha_0}^-)  \supset \dots  \supset \gamma_{n}^{-(m-1)} \Nc_{2\epsilon_0}(U_{\alpha_0}^-)  \supset\gamma_{n}^{-m}\Nc_{2\epsilon_0}(U_{\alpha_0}^-)\] and $\diam_\infty(\gamma_{n}^{-m}\Nc_{2\epsilon_0}(U_{\alpha_0}^-)) \leqslant \frac{1}{10^{mj-1}}\epsilon_0$. As $\partial_\infty\Gamma$ is compact, \[\bigcap_{m\in \Nb} \gamma_{n}^{-m} \overline{\Nc_{2\epsilon_0}(U_{\alpha_0}^-)}\] is the singleton $\{\gamma_n^{-}\}$. Then $d_\infty(y_0,\gamma_n^-)\leqslant 3\epsilon_0$.

    Similarly, we have $(\gamma_i\gamma_{n}\gamma_i^{-1})^{-m} \Nc_{2\epsilon_0}(U_{\alpha_i}^-)$ is also a decreasing sequence whose $d_\infty$-diameter converges to $0$, then \[\bigcap_{m\in \Nb} (\gamma_i\gamma_{n}\gamma_i^{-1})^{-m} \overline{\Nc_{2\epsilon_0}(U_{\alpha_i}^-)} = \{ (\gamma_i\gamma_{n}\gamma_i^{-1})^- = \gamma_i \gamma_n^-\}\quad \text{ and } \quad d_\infty(y_i,\gamma_i\gamma_n^-)\leqslant 3\epsilon_0~.\]

    Following Remark~\ref{remark: reverse flow transversal}, we apply the same argument for the reversed flow and obtain that \[\bigcap_{m\in \Nb} (\gamma_i\gamma_{n}\gamma_i^{-1})^{m} \overline{\Nc_{2\epsilon_0}(U_{\alpha_i}^+)} = \{ (\gamma_i\gamma_{n}\gamma_i^{-1})^+ = \gamma_i \gamma_n^+\}\quad \text{ and } \quad d_\infty(x_i,\gamma_i\gamma_n^+)\leqslant 3\epsilon_0~.\] Therefore we deduce that for each $0\leqslant i \leqslant n$, \[(\gamma_n^+,\gamma_n^-,t_{\alpha_0}+ c(\gamma_n^+,\gamma_n^-,\gamma_i))\in \gamma_i \Nc_{2\epsilon_0}(K_{\alpha_i})~.\] By the choice of $\epsilon_0$, we have that the Hausdorff distance between \[\gamma_i\phi^{[0,t_i]}(v_i)\quad \text{and} \quad (\gamma_n^+,\gamma_n^-,[t_{\alpha_0}- c(\gamma_n^+,\gamma_n^-,\gamma_{i}^{-1}), t_{\alpha_0} - c(\gamma_n^+,\gamma_n^-,\gamma_{i+1}^{-1})])\] (that is, each black arrow and the nearby blue arrow in Figure~\ref{figure: closing lemma}) is bounded by $R_{\Haus}$. By taking the union we get that the Hausdorff distance between \[\bigcup_{i=0}^{n-1}\gamma_i\phi^{[0,t_i]}(v_i)\quad \text{and} \quad (\gamma_n^+,\gamma_n^-,[t_{\alpha_0}, t_{\alpha_0} + c(\gamma_n^+,\gamma_n^-,\gamma_n)])\] is bounded by $R_{\Haus}$.
\end{proof}

\subsection{Density of multigeodesics}\label{section: Density of multigeodesics}

A geodesic current on $\Gamma$ is a locally finite, $\Gamma$-invariant, $\iota$-invariant measure on $\partial_\infty^{(2)}\Gamma$. For a loxodromic element $\gamma\in \Gamma$, recall that we denote by $\widehat{\gamma}$ the counting measure of the $\Gamma$-orbit of the axis of $\gamma$. More concretely, for any $E\subset \partial_\infty^{(2)}\Gamma$, \[\widehat{\gamma}(E) = q_\gamma \cdot \# \Big(\big(\Gamma (\gamma^+,\gamma^-)\cup \Gamma(\gamma^-,\gamma^+)\big) \cap E\Big)~,\]
where \[q_\gamma= \#(\mathrm{Stab}_\Gamma\{\gamma^+,\gamma^-\}/\langle\gamma\rangle)~,\] is the multiplicity of the axis.
We denote by $\Curr(\Gamma)$ the set of geodesic currents and by $\PCurr(\Gamma)$ the quotient of $\Curr(\Gamma)$ by $\Rb_{>0}$ scaling, and we abuse the notation $\widehat{\gamma}$ as element of $\PCurr(\Gamma)$.

\begin{remark}
    The total mass of the measure induced by $\widehat\gamma\otimes dt$ on $\Gamma\backslash G(\Gamma)$ is uniformly comparable to $\norm{\gamma}$. Therefore, if $\widehat{\gamma}_n \to \mu$ in $\PCurr(\Gamma)$, then up to subsequence and a choice of nonzero representative $\mu_1$ of $\mu$ in $\Curr(\Gamma)$, we have \[\frac{\widehat{\gamma}_n}{\norm{\gamma_n}} \to\mu_1 \quad \text{as } n \to \infty \quad \text{in } \Curr(\Gamma)~.\] We always use this normalization when interpreting convergence of projective measures as convergence of measures.
\end{remark}

We fix a global transversal $K$ and let $\epsilon_0$ be given by Notation~\ref{notation: the choice of constants before closing lemma}.

\begin{definition}
    Let $\epsilon\leqslant \epsilon_0$ and let $\Ac$ be a $\epsilon$-partition of $K$. Suppose $(\alpha_i)_{i=0}^k\subset \Ac$. We say $v\in \Gc(\Gamma)$ follows the \emph{pattern} $(\alpha_i)_{i=0}^k$ if 
    \begin{enumerate}
        \item there exists $0= t_0 < t_1<t_2<\dots< t_k$ such that $\phi^{t_i}(v)\in \Gamma K_{\alpha_i}$ for each $0\leqslant i \leqslant k$;
        \item $\phi^t(v)\not\in \Gamma K$ for any $t\in (0,t_k)\setminus\{t_1,t_2,\dots,t_{k-1}\}$.
    \end{enumerate}
    We say $\wh v \in \wh \Gc(\Gamma)$ follows the \emph{pattern} $(\alpha_i)_{i=0}^k$ if a (any) lift of $\wh v$ in $\Gc(\Gamma)$ via the projection follows the pattern $(\alpha_i)_{i=0}^k$.
\end{definition}

Up to further shrinking $\epsilon_0$, the following fact holds. Let $\epsilon\leqslant \epsilon_0$ and let $\Ac$ be a $\epsilon$-partition of $K$. For any $(\alpha_i)_{i=0}^k\subset \Ac$, there exists $(\gamma_i)_{i=1}^k\subset \Gamma$ depending only on $(\alpha_i)$ such that any vector $v\in K_{\alpha_0}$ following the pattern $(\alpha_i)_{i=0}^k\subset \Ac$ has successive
hits to $\Gamma K$ in the order $ K_{\alpha_0}, \gamma_1  K_{\alpha_1},\dots,  \gamma_k K_{\alpha_k}$.

\begin{definition}\label{definition: track of geodesic current on partition}
    Let $\epsilon\leqslant \epsilon_0$ and let $\Ac$ be a $\epsilon$-partition of $K$. Let $\mu \in \Curr(\Gamma)$. The \emph{track} of $\mu$ on $\Ac$ is the oriented weighted graph \[(\Graph(\Ac),w_{(\mu,\Ac)})\]consisting of the following data,
    \begin{enumerate}
        \item $\Graph(\Ac)$ is an oriented graph with set of vertices $\Ac$ and set of arrows $\Ac\times \Ac$;
        \item Each arrow $(\alpha,\beta) \in \mathcal A \times \mathcal A$ is assigned a weight \[w_{(\mu,\Ac)}(\alpha\beta) = \mu(\{(x,y)\in U_{\alpha}^+\times U_{\alpha}^- : \exists\ t>0, (x,y,t_\alpha+t) \in \Gamma K_\beta, \] \[\text{ and } \ (x,y,(t_\alpha,t_\alpha+t))\cap \Gamma K = \emptyset)\}~.\]
    \end{enumerate}
\end{definition}

In other words, for each arrow $(\alpha,\beta) \in \mathcal A \times \mathcal A$, $w_{(\mu,\Ac)}(\alpha\beta)$ is the $\mu$-measure of the set of geodesics whose endpoints lie in $U_\alpha^{+} \times U_\alpha^{-}$ and the first return to $\Gamma K$ after the time at $t_\alpha$ occurs in $\Gamma K_\beta$, i.e., follows the pattern $[(\alpha,\beta)]$.

\begin{definition}
    We follow the notation as in Definition~\ref{definition: track of geodesic current on partition}. For $(\alpha_i)_{i=0}^{m-1}\subset \Ac$, we set $\alpha_m = \alpha_0$ and we call a function $c: \Ac \times \Ac \to \Rb_{\geqslant 0}$ a weighted loop in $\Graph(\Ac)$ with vertices $(\alpha_i)_{i=1}^{m-1}$ if $c(\alpha_i,\alpha_{i+1}) = 1$ for any $0\leqslant i\leqslant m-1$, and takes value $0$ on all the other arrows.
\end{definition}

Some direct observations are that

\begin{observation}\label{observation: system of weight equations}
    If $\Ac$ is an $\epsilon$-partition of $K$ for some $\epsilon\leqslant \epsilon_0$ and $\mu \in \Curr(\Gamma)$, then $w_{(\mu,\Ac)}$ satisfies the following finitely many homogeneous linear equations, \[\forall \alpha\in \Ac~, \quad \sum_{\beta\in \Ac} w_{(\mu,\Ac)}(\alpha\beta) = \sum_{\beta\in \Ac} w_{(\mu,\Ac)}(\beta\alpha)~.\]
\end{observation}

\begin{observation}\label{observation: rational approximation of weights}
    We follow the notations as in Observation~\ref{observation: system of weight equations}. For any $\delta>0$, there exists $q: \Ac \times \Ac \to \Qb_{\geqslant 0}$ such that
    \begin{enumerate}
        \item It approximates $w_{(\mu,\Ac)}$, that is, \[\sum_{\alpha\in\Ac}\sum_{\beta\in\Ac}\abs{q(\alpha\beta)-w_{(\mu,\Ac)}(\alpha\beta)} \leqslant \delta~;\]
        \item It solves the same equations as $w_{(\mu,\Ac)}$, that is, \[\forall \alpha\in \Ac~, \quad \sum_{\beta\in \Ac} q(\alpha\beta) = \sum_{\beta\in \Ac} q(\beta\alpha)~;\]
        \item It takes value $0$ at the same place as $w_{(\mu,\Ac)}$, that is, $q(\alpha\beta) = 0 $ if and only if $w_{(\mu,\Ac)}(\alpha\beta) = 0$.
    \end{enumerate}
\end{observation}

\begin{observation}\label{observation: loop decomposition}
    For any map $q: \Ac \times \Ac \to \Qb_{\geqslant 0}$ such that \[\forall \alpha\in \Ac, \ \sum_{\beta\in \Ac} q(\alpha\beta) = \sum_{\beta\in \Ac} q(\beta\alpha)~,\]
    there exist finitely many weighted loops $(c_j)_{j=1}^k$ in $\Graph(\Ac)$ and positive rationals $(a_j)_{j=1}^k$ such that \[q = \sum_{j=1}^k a_jc_j~.\]
\end{observation}

Let $\epsilon\leqslant \epsilon_0$. Suppose $\Ac$ is an $\epsilon$-partition of $K$. We now refine the partition $\Ac$ by the following process. For each $k$–tuple $(\alpha_i)_{i=0}^k \subset \Ac$, we define a subset of $K$ by \[K[(\alpha_i)_{i=0}^k] = \{v\in K_{\alpha_0} : v \text{ follows the pattern }(\alpha_i)_{i=0}^k\}~.\] In particular for each $\alpha \in \Ac$, $K_\alpha = K[(\alpha)]$. Then we have \[K[(\alpha_i)_{i=0}^k] = \bigsqcup_{\alpha_{k+1}\in \Ac} K[(\alpha_i)_{i=0}^{k+1}]~.\]  The collection \[ \{ K[(\alpha_i)_{i=0}^k] : (\alpha_i)_{i=0}^k \subset \Ac \}~\] gives a partition of $K$. We briefly denote this partition by \[K= \bigsqcup_{\alpha^{(k)}\in \Ac^{(k)}} K_{\alpha^{(k)}}~,\] or simply the index set $\Ac^{(k)}$. We denote $\Ac = \Ac^{(0)}$, then $\Ac^{(k)}$ is a refinement of $\Ac^{(k-1)}$ for any $k\in \Nb$. For any $\alpha^{(k)}\in \Ac^{(k)}$, we denote by $\alpha^{(k-1)}\in \Ac^{(k-1)}$ the index such that $K_{\alpha^{(k)}}\subset K_{\alpha^{(k-1)}}$.

\begin{definition}
    With the construction above, we call $\Ac^{(k)}$ the \emph{$k$-th returning refinement} of $\Ac$.
\end{definition}

Then the following lemma is direct by our construction of the returning refinements.

\begin{lemma}\label{lemma: refinement determines pattern}
    Let $\mu\in \Curr(\Gamma)$, $\epsilon\leqslant \epsilon_0$, and $\Ac$ be a $\epsilon$-partition of $K$. Suppose $\{\alpha_i^{(k)}\}_{i = 0}^k \subset \Ac^{(k)}$ is such that for any $0\leqslant i \leqslant k$, $w_{(\mu,\Ac^{(k)})}(\alpha_{i}^{(k)}\alpha_{i+1}^{(k)})>0$. Then for any $v\in K_{\alpha_0^{(k)}}$,
    \begin{enumerate}
        \item there exists $0=t_0<t_1<\dots<t_k$ such that $\forall 0\leqslant i\leqslant k$, $\phi^{t_i}(v)\in \Gamma K_{\alpha_i^{(k-i)}}$;
        \item $\phi^{t}(v)\notin \Gamma K$ if $t\in (0,t_k)\setminus\{t_1,\dots,t_{k-1}\}$.
    \end{enumerate} 
\end{lemma}

\begin{proof}[Sketch of the proof.]
    an element $K_{\alpha_0^{(k)}}$ of the $k$-th returning refinement records the first $k$ successive returns of an orbit to $\Gamma K$. Thus, for any $v\in K_{\alpha_0^{(k)}}$, since $w_{(\mu,\Ac^{(k)})}(\alpha_{0}^{(k)}\alpha_{1}^{(k)})>0$, then the first return of $v$ to $\Gamma K$ occurs in a prescribed translate $\gamma_1 K_{\alpha_1^{(k-1)}}$. After translating back by $\gamma_1^{-1}$, the returned point lies in $K_{\alpha_1^{(k-1)}}$, and this set now records the next $k-1$ returns. Repeating the same argument inductively, one obtains group elements $\gamma_1,\dots,\gamma_k$ and return times $0=t_0<t_1<\cdots<t_k$ such that \[\phi^{t_i}(v)\in \gamma_i K_{\alpha_i^{(k-i)}} \quad \text{ for every }0\le i\le k~,\] and there are no other intersections with $\Gamma K$ between these return times.
\end{proof}

\begin{definition}[Multigeodesics]\label{definition: multigeodesics}
    We call a finite collection of loxodromic elements $\{\gamma_j\}_{j=1}^k \subset \Gamma$ a multigeodesic. A weighted multigeodesic $\Y$ is the data of a multigeodesic $\{\gamma_j\}_{j=1}^k \subset \Gamma$ and positive reals $a_1,a_2,\dots, d_k$ associated to each element of $\{\gamma_j\}$. We write it as a formal sum $\Y= \sum_{i=1}^k a_i \gamma_i$.
\end{definition}

\begin{definition}[Strong convergence]\label{definition: strong convergence}
    We say a sequence of weighted multigeodesics $\Y_n = \sum_{j=1}^{k_n}a_{nj}\gamma_{nj}$
    \emph{converges} to $\mu \in \Curr(\Gamma)$ if with respect to the weak* topology, \[\lim_{n \to \infty} \Y_n :=\lim_{n \to \infty}\sum_{j=1}^{k_n}a_{nj}\widehat{\gamma}_{nj} = \mu~;\]
    We say $\Y_n$ \emph{strongly converges} to $\mu$ if moreover with respect to the Hausdorff distance, \[\lim_{n \to \infty} \bigcup_{j=1}^{k_n}\Gamma(\gamma_{nj}^+,\gamma_{nj}^-,\Rb) = (\supp \mu) \times \Rb\subset \Gc(\Gamma)~.\]
\end{definition}

\begin{proof}[Proof of Theorem~\ref{theorem: density of multigeodesics}]
    We fix a global transversal $K$ of $\Gc(\Gamma)$. Suppose $\mu \in \Curr(\Gamma)$. We fix a large enough $T_3>0$ such that we can choose integers $M_2$ and $M_3$ satisfying \[M_2 \geqslant \frac{T_2}{T_0}~, \quad M_3 \leqslant \frac{T_3}{T_1}-1~, \quad \text{and} \quad M_3 \geqslant 2M_2 +1~.\] 
    By Lemma~\ref{lemma: flow time quasi isometric to group action} and Lemma~\ref{lemma: existence global transversal}, if $v\in K$, $t>0$ and $\gamma\in\Gamma$ are such that $\phi^t(v) \in \gamma K$, then \begin{enumerate}
        \item $\#(\phi^{[0,t]}(v) \cap \Gamma K) \geqslant M_2 \ \Rightarrow \ t\geqslant T_2$.
        \item $\#(\phi^{[0,t]}(v) \cap \Gamma K) \leqslant M_3 \ \Rightarrow \ t \leqslant T_3$.
    \end{enumerate}
    
    For any $R_{\Haus}>0$, together with the constant $T_3$ above, let $\epsilon_0$ be determined by Theorem~\ref{theorem: a closing lemma}. Let $\epsilon \leqslant \epsilon_0$ and let $\Ac$ be a $\epsilon$-partition of $K$. Let $N$ be a fixed integer in between $2M_2 +1$ and $M_3$.

    By Observation~\ref{observation: system of weight equations}, we have $\forall \alpha^{(N)}\in \Ac^{(N)}$, \[\ \sum_{\beta^{(N)}\in \Ac^{(N)}} w_{(\mu,\Ac^{(N)})}(\alpha^{(N)}\beta^{(N)}) = \sum_{\beta^{(N)}\in \Ac^{(N)}} w_{(\mu,\Ac^{(N)})}(\beta^{(N)}\alpha^{(N)})~,\] and by Observation~\ref{observation: rational approximation of weights}, for any $\delta> 0$ there exists $q^{(N)}: \Ac^{(N)} \times \Ac^{(N)} \to \Qb_{\geqslant 0}$ that solves the same equations as $w_{(\mu,\Ac^{(N)})}$, \[\sum_{\alpha^{(N)}\in\Ac^{(N)}}\sum_{\beta^{(N)}\in\Ac^{(N)}}\abs{q^{(N)}(\alpha^{(N)}\beta^{(N)})-w_{(\mu,\Ac^{(N)})}(\alpha^{(N)}\beta^{(N)})} \leqslant \delta~,\] and $q^{(N)}$ takes value $0$ at the same places as $w_{(\mu,\Ac^{(N)})}$. Let $q: \Ac \times \Ac \to \Qb_{\geqslant 0}$ be the weight obtained from $q^{(N)}$ by forgetting the refinement, that is, \[q(\alpha\beta) = \sum_{\alpha^{(N)}\subset \alpha}\sum_{\beta^{(N)}\subset \beta}q^{(N)}(\alpha^{(N)}\beta^{(N)})~.\] Then \[\sum_{\alpha \in\Ac}\sum_{\beta\in\Ac}\abs{q(\alpha\beta)-w_{(\mu,\Ac)}(\alpha\beta)} \leqslant \delta~.\]
    
    By Observation~\ref{observation: loop decomposition}, there exist finitely many weighted loops $(c_j)_{j=1}^k$ in $\Graph(\Ac^{(N)})$ and $(a_j)_{j=1}^k\subset \Qb_{>0}$ such that \[q^{(N)} = \sum_{j=1}^k a_jc_j~.\] We then approximate each weighted loop $c = c_j$ from $(c_j)_{j=1}^k$ by closed flow lines. Suppose $c$ has vertices $(\alpha_0^{(N)} = \alpha_{m}^{(N)},\alpha_1^{(N)},\dots,\alpha_{m-1}^{(N)})$ with each $\alpha_i^{(N)}\in \Ac^{(N)}$.\\

    \noindent {\bf Case 1.}
    If $m\geqslant M_2$, let $0=N_0 <N_1<\dots < N_r = m,$ be such $M_2\leqslant N_{i+1}-N_{i} \leqslant N\leqslant M_3$ for each $0\leqslant i\leqslant r-1$. Then by Lemma~\ref{lemma: refinement determines pattern}, any $v_i\in K_{\alpha_{N_i}^{(N)}}$ follows the pattern \[(\alpha_{N_i}^{(0)},\alpha_{N_i+1}^{(0)},\dots, \alpha_{N_{i+1}}^{(0)})~.\] Then by Theorem~\ref{theorem: a closing lemma}, there exists $\gamma_j\in \Gamma$, such that $(\gamma_j^+,\gamma_j^-,t_{\alpha_{0}})$ follows the pattern exactly the vertices of $c_j$.

    \noindent {\bf Case 2.}
    If $m\leqslant M_2$, let $r\in \Nb$ such that $M_2\leqslant rm\leqslant N$, we write \[\alpha_{i+pm}^{(N)} = \alpha_i^{(N)}, \ \forall 0\leqslant i\leqslant m-1\text{ and }1\leqslant p \leqslant r~.\] Then by Lemma~\ref{lemma: refinement determines pattern}, any $v\in K_{\alpha_0^{(N)}}$ follows the pattern \[(\alpha_{0}^{(0)},\alpha_{1}^{(0)},\dots, \alpha_{rm}^{(0)})~.\] Then by Theorem~\ref{theorem: a closing lemma}, there exists $\gamma_j\in \Gamma$, such that $(\gamma_j^+,\gamma_j^-,t_{\alpha_{0}})$ follows the pattern precisely the vertices of repeating $c_j$ for $r$ times. In this case we replace $a_j$ by $\frac{1}{r}a_j$.
    
    Therefore we get a geodesic current \[\nu = \sum_{j=1}^k a_j \widehat{\gamma}_j\] which has track exactly $(\Graph(\Ac),q)$.

    For any two sequence of positive reals $(\epsilon_n)_{n\in \Nb}$ and $(\delta_n)_{n\in \Nb}$ that converge to $0$, let $(\Ac_n)_{n\in \Nb}$ be a sequence of partitions of $K$ such that for each $n\in \Nb$, $\Ac_n$ is an $\epsilon_n$-partition and $\Ac_{n+1}$ is a refinement of $\Ac_n$, and let \[\nu_n  = \sum_{j=1}^{k_n} a_{nj} \widehat{\gamma}_{nj}\] be constructed as the $\nu$ above for $\Ac = \Ac_n$, $\epsilon = \epsilon_n$ and $\delta = \delta_n$. Then by Theorem~\ref{theorem: a closing lemma}, the multigeodesic $(\gamma_{nj})_{j=1}^{k_n}$ is such that \[\lim_{n \to \infty} \bigcup_{j=1}^{k_n}\Gamma(\gamma_{nj}^+,\gamma_{nj}^-,\Rb) = (\supp \mu) \times \Rb\subset \Gc(\Gamma)~.\] The convergence $\nu_n \to \mu$ as $n\to \infty$ with respect to the weak* topology follows from a standard measure-theoretic argument.
\end{proof}

\section{Ending currents}\label{section: ending currents}

We develop the theory of ending currents in this section. In Section~\ref{subsection: notions of ending current}, we will add a few notes and details to the definitions in the introduction of the various ending currents and their supports and prove Theorem~\ref{main theorem: the relations among ending currents}. In Section~\ref{subsection: ending currents of hyperbolic group extensions}, we study ending currents for hyperbolic group extensions, prove Theorem~\ref{main theorem: limit set of convex-cocompact representations}, and construct the example that verifies Theorem~\ref{main theorem: triangles limit to a polygon}. In Section~\ref{subsection: ending currents of geometrically finite representations}, we discuss how ending currents characterize geometrically finite actions and prove Theorem~\ref{main theorem: geometric finiteness and ending currents}. Finally, in Section~\ref{section: geometrically infinite example}, we construct the geometrically infinite example for Theorem~\ref{main theorem: geometrically infinite example}.

\subsection{Ending currents and their relations}\label{subsection: notions of ending current}

Recall again that $\Gamma$ is a $\delta$-hyperbolic group and we fix a discrete action $\rho:\Gamma \to \Isom(X)$ with finite kernel for a $\delta$-hyperbolic space $X$. 

We remark first that we could have added an additional definition, a contracting ending current that takes limits of multigeodesics instead of single geodesics, namely:
\[\hEmcont = \left\{ \lim_{n \to \infty}\widehat\Y_n\in \PCurr(\Gamma): (\Y_n)\subset\MM(\Gamma), \quad \lim_{n \to   \infty}\frac{\ell_\rho(\Y_n)}{\norm{\Y_n}} = 0\right\}~.\]

However, in our setting they coincide:
\begin{lemma}\label{lemma: contracting definitions agree}
    $\hEcont = \hEmcont$.
\end{lemma}

\begin{proof}
    It is clear that $\hEcont(\rho) \subset \hEmcont$ by definition. 
    It remains to show that $\hEmcont(\rho) \subset \hEcont(\rho)$.

    Let $\Y_n = \sum_{j=1}^{k_n} a_{nj}\gamma_{nj}$
    be a defining sequence for $\hEmcont(\rho)$. We fix a sequence of constants $0<\epsilon_n<1$ that converges to $0$. For each $n\in \Nb$, it follows from a standard local-to-global argument that up to replacing each $\gamma_{nj}$ by a conjugate or a conjugate of its inverse, for large enough integers $\{m_{nj}\}_{j=1}^{k_n}$, we have \[\gamma_n = \gamma_{n1}^{m_{n1}}\gamma_{n2}^{m_{n2}}\dots \gamma_{n k_n}^{m_{n k_n}}\] is loxodromic and satisfies 
    \begin{equation}\label{remark: single geodesics limit, equation: word metric}
        \norm{\gamma_n} \geqslant (1-\epsilon_n) \sum_{j=1}^{k_n}m_{nj}\norm{\gamma_{nj}}~.
    \end{equation}
    On the other hand, for large enough integers $\{m_{nj}\}_{j=1}^{k_n}$,
    \begin{equation}\label{remark: single geodesics limit, equation: stable length}
        \ell_\rho(\gamma_n) \leqslant \sum_{j=1}^{k_n} d_X(o,\rho(\gamma_{nj}^{m_{nj}})o) \leqslant \sum_{j=1}^{k_n}m_{nj}(\ell_\rho(\gamma_{nj})+\epsilon_n)~.
    \end{equation}
    
    We fix $\{m_{nj}\}_{j=1}^{k_n}\subset \Nb$ such that \eqref{remark: single geodesics limit, equation: word metric} and \eqref{remark: single geodesics limit, equation: stable length} hold and \[\sum_{j=1}^{k_n} \abs{\frac{m_{nj}}{\sum_{j=1}^{k_n}m_{nj}} - \frac{a_{nj}}{\sum_{j=1}^{k_n}a_{nj}}} \leqslant \epsilon_n~.\] Then we have \[\lim_{n \to \infty}\frac{\ell_\rho(\gamma_n)}{\norm{\gamma_n}} 
    \leqslant \lim_{n \to \infty}\left(\frac{\epsilon_n}{1-\epsilon_n} + \frac{1}{1-\epsilon_n} \left(\frac{\sum_{j=1}^{k_n}a_{nj}\ell_\rho(\gamma_{nj})} {\sum_{j=1}^{k_n}a_{nj}\norm{\gamma_{nj}} }\right)\right) = 0\]
    and $\displaystyle\lim_{n \to \infty}\widehat\Y_n = \lim_{n\to \infty} \widehat{\gamma}_n$. Therefore $\hEmcont(\rho) \subset \hEcont(\rho)$.
\end{proof}

As in the introduction we define \[\Ebdd(\rho),\ \Econt(\rho),\ \Eesc(\rho), \text{ and } \Emesc(\rho)\] as the unions of supports of \[\hEbdd(\rho),\ \hEcont(\rho),\ \hEesc(\rho) \text{ and } \hEmesc(\rho),\] respectively. 
Notice that $\hEbdd(\rho)$, $\hEcont(\rho)$, $\hEesc(\rho)$ and $\hEmesc(\rho)$ are closed subsets of $\PCurr(\Gamma)$, while $\Ebdd(\rho)$, $\Econt(\rho)$, $\Eesc(\rho)$ and $\Emesc(\rho)$ are not necessarily closed in $\partial_\infty^{(2)}\Gamma$.

The definitions of $\Ehausesc(\rho)$ and $\Lambda(\rho)$, given in the introduction, are closely related but we remark that the condition on a defining sequence $\{\gamma_n\}_{n\in \Nb}\subset \Gamma$ for $\Lambda(\rho)$ is in general weaker than the one for $\Ehausesc(\rho)$.  The former requires that the axis of $[\rho(\gamma_n)]$ escape from any compact set in $X$,  whereas the latter requires that the axis of $[\rho(\gamma_n)]$ escape from the $\Gamma$-orbit of any compact set in $X$. Thus we have $\Ehausesc(\rho)\subset \Lambda(\rho)$.

It also directly follows from the definitions that $\hEbdd(\rho) \subset \hEcont(\rho)$, $\hEesc(\rho) \subset \hEmesc(\rho)$.

\begin{proposition}\label{proposition: bounded ending currents contained in far away currents}
  In the setting above, $\hEbdd(\rho) \subset \hEesc(\rho)$ and hence $\Ebdd(\rho) \subset \Eesc(\rho)$.
\end{proposition}

\begin{proof}
    We show that for any sequence $(\gamma_n)_{n\in \Nb} \subset \Gamma$ such that $(\ell_\rho(\gamma_n))_{n\in \Nb}$ is uniformly bounded by a constant $L\geqslant 0$ and $\lim_{n \to \infty}\norm{\gamma_n} = +\infty$, we must have $\lim_{n \to \infty} r(\gamma_n) = +\infty$.

    Suppose otherwise, up to subsequence, $(r(\gamma_n))_{n\in \Nb}$ is uniformly bounded by $M\geqslant 0$, then for each $n\in \Nb$, there exists $\gamma_n'$, a conjugate of $\gamma_n$, such that $[\rho(\gamma_n')]$ passes through a point $a_n$ in the neighborhood of $o$ of radius $M$ in $X$. Then \[d_X(o,\rho(\gamma_n')o) \leqslant d_X(o,a_n)+d_X(a_n,\rho(\gamma_n')a_n)+d_X(\rho(\gamma_n')a_n,\rho(\gamma_n')o) \leqslant 2M+L~.\] On the other hand, $\lim_{n \to \infty}\norm{\gamma_n} = +\infty$ implies that \[\{\text{conjugacy class of }\gamma_n: n\in \Nb\}\] is an infinite set and hence $(\gamma_n')_{n\in \Nb}$ is an infinite set. However, the neighborhood of $o$ of radius $2M+L$ in $X$ contains $\{\gamma_n'o: n\in \Nb\}$, which contradicts that $\rho(\Gamma)$ acts properly discontinuously on $X$.
\end{proof}

\begin{proposition}\label{proposition: far away currents contained in ending currents}
    In the setting above, $\hEmesc(\rho) \subset \hEcont(\rho)$ and hence $\Emesc(\rho) \subset \Econt(\rho)$.
\end{proposition}

The proposition is direct by the lemma below.

\begin{lemma}
    If $(\gamma_n) \subset \Gamma$ is a sequence such that \[\lim_{n \to \infty} r(\gamma_n) = +\infty~,\] then \[\lim_{n \to \infty}\frac{\ell_\rho(\gamma_n)}{\norm{\gamma_n}} = 0~.\]
\end{lemma}

\begin{proof}
    We may assume without loss of generality that $\gamma_n$ is cyclically reduced, that is, an axis of $\gamma_n$ in the Cayley graph of $\Gamma$ passes through $\id$ and $\rho(\gamma_n)$ is loxodromic for each $n\in \Nb$. Then $\norm{\gamma_n} = \abs{\gamma_n}$. Let $L_n$ be the piecewise geodesic in $X$ that passing through $(\rho(\gamma_n^k)o)$ for all $k\in \Zb$.
    
    Suppose otherwise up to subsequence there exists a constant $C_1>0$ such that $\frac{\ell_\rho(\gamma_n)}{\norm{\gamma_n}}  \geqslant 2C_1$. Since $\Gamma$ is finitely generated, there exists a constant $C_2>0$, such that \[C_1\abs{\gamma_n} \leqslant \ell_\rho(\gamma_n) \leqslant d_X(o,\rho(\gamma_n)o)\leqslant C_2 \abs{\gamma_n}~.\] Then $L_n$ is a quasi-geodesic uniform for each $n\in \Nb$. By applying the Morse lemma to $L_n$ and $[\rho(\gamma_n)]$ we have that $d_X(o, [\rho(\gamma_n)])$ is bounded above uniformly, which contradicts that $r(\gamma_n)$ diverges.
\end{proof}

We have completed the proof of Theorem~\ref{main theorem: the relations among ending currents}~(1) and (2).\\

\begin{definition}[Cannon--Thurston map]\label{definition: Cannon--Thurston map}
    A \emph{Cannon--Thurston map} of $\rho$ is a continuous, $\rho$-equivariant map $\xi_\rho : \partial_\infty \Gamma \to \partial_\infty X$ that continuously extends a (any) orbit map $\Gamma\to X$ on the boundary. When a Cannon--Thurston map $\xi_\rho$ exists, we denote \[\Omega(\rho) = \{(x,y)\in \partial_\infty^{(2)}\Gamma: \xi_\rho(x) = \xi_\rho(y)\}~.\]
\end{definition}

The following proposition is stated originally for hyperbolic group extensions but extends trivially to our setting.

\begin{proposition}[{\cite[Lemma~2.2]{mitra1999theorem}}]
    Suppose $\rho$ admits a Cannon--Thurston map $\xi_\rho$, then $\Lambda(\rho) = \Omega(\rho)$.
\end{proposition}

In the rest of this section, we show that when a Cannon--Thurston map exists, $\hEmesc(\rho) = \hEcont(\rho)$ and $\Emesc(\rho)\subset \Ehausesc(\rho) \subset \Omega(\rho)$. For any given $\mu\in \Curr (\Gamma)$, we denote by $\mu_*$ the restriction on $\partial_\infty^{(2)}X$ of the push-forward measure of $\mu$ on $\partial_\infty^2 X$ via $\xi_\rho\times \xi_\rho$. Note that $\mu \otimes \de t$ 
is a $\Gamma$-invariant, $\phi$-invariant measure on $\Gc(\Gamma)$. 

\begin{lemma}\label{lemma: ending currents contained in Cannon--Thurston gluing}
    Suppose $\xi_\rho$ is a Cannon--Thurston map of $\rho$. Then $\Econt(\rho)\subset \Omega(\rho)$.
\end{lemma}

\begin{proof}
    We first show that for any compact set $K\subset \partial_\infty^{(2)}\Gamma\setminus\Omega(\rho)$, there exists a constant $C_K>0$ such that for any $\gamma \in \Gamma\setminus\{\id\}$, \begin{equation}\label{eq: measure bounded by length}
        \widehat{\gamma}(K)\leqslant C_K\ell_\rho(\gamma)~.
    \end{equation}

    Since $\xi_\rho$ is continuous and $K$ is compact in $\partial_\infty^{(2)}\Gamma\setminus\Omega(\rho)$, we have $(\xi_\rho\times\xi_\rho)(K)$ is a compact subset of $\partial_\infty^{(2)}X$. Then there exists a compact set $F_0\subset X$ such that every geodesic with endpoints in $(\xi_\rho\times\xi_\rho)(K)$ passes through $F_0$. For $\gamma \in \Gamma\setminus\{\id\}$, if $\rho(\gamma)$ is parabolic, then $(\gamma^+,\gamma^-)\in \Omega(\rho)$ and $\widehat{\gamma}(K) = 0$ and $\ell_\rho(\gamma) = 0$. We then assume $\rho(\gamma)$ is loxodromic. Let $L$ be a geodesic joining $\rho(\gamma)^+$ and $\rho(\gamma)^-$ that is $\rho(\gamma)$-invariant. Let $p\in L$ and let $L_n$ be the segment of $L$ from $p$ to $\rho(\gamma)^np$.

    Let $F_1$ be the neighborhood of $F_0$ of radius $1$. 
    Since $\rho(\Gamma)$ acts on $X$ properly discontinuously, $m = \#\{\eta \in \Gamma : \rho(\eta) F_1 \cap F_1 \ne \emptyset\}$ is finite. Then (for $n$ large enough)
    \[n \widehat{\gamma}(K) \leqslant \#\left\{\eta\in\Gamma: L_n\cap\rho(\eta)F_0\neq\varnothing\right\}\leqslant m\,d_X(p,\rho(\gamma)^np)~,\]
    where the second inequality follows from the fact that every time $L_n$ passes through $\rho(\eta)F_0$ for some $\eta$, it spends length at least $1$ in $\rho(\eta)F_1$ and the translates of $F_1$ cover any point of $L_n$ with multiplicity at most $m$. 
    Therefore we have \[\widehat{\gamma}(K)\leqslant \lim\limits_{n\to \infty}\frac{m}{n} d_X(p,\rho(\gamma)^n p) = m \ell_\rho(\gamma)~,\] which completes the proof of \eqref{eq: measure bounded by length}.

    Suppose $\mu \in \hEcont(\rho)$ and $(\gamma_n)\subset \Gamma$ is such that $\widehat \gamma_n \to \mu$ in $\PCurr(\Gamma)$ and $\frac{\ell_\rho(\gamma_n)}{\norm{\gamma_n}} \to 0$ as $n\to \infty$. Then there is a nonzero lift $\mu_1$ of $\mu$ in $\Curr(\Gamma)$ such that \[\frac{\widehat{\gamma}_n }{\norm{\gamma_n}} \to\mu_1 \quad \text{as } n \to \infty \quad \text{in } \Curr(\Gamma)~.\]
    For any nonnegative continuous function $f$ with compact support in $\partial_\infty^{(2)}\Gamma\setminus\Omega(\rho)$, we set $K=\supp f$. Then \[0 \leqslant \int_K f \de\mu_1 =\lim_{n\to\infty} \int_K f \frac{1}{\norm{\gamma_n}}\de \widehat{\gamma}_n\leqslant \norm{f}_\infty \lim_{n\to\infty}  \int_K  \frac{1}{\norm{\gamma_n}}\de \widehat{\gamma}_n \]\[=  \norm{f}_\infty  \lim_{n\to\infty}\frac{1}{\norm{\gamma_n}} \widehat{\gamma}_n(K)\leqslant C_K\norm{f}_\infty \lim_{n\to\infty} \frac{\ell_\rho(\gamma_n)}{\norm{\gamma_n}}=0~.\] Since $\Omega(\rho)$ is closed, this implies that $\supp \mu \subset\Omega(\rho)$. Therefore $\Econt(\rho)\subset\Omega(\rho)$.
\end{proof}



\begin{lemma}\label{lemma: sequences of geodesics converges into Cannon--Thurston gluing escapes}
    Suppose $\xi_\rho$ is a Cannon--Thurston map of $\rho$. If $(\alpha_n)\subset \Gamma$ is a sequence such that $\sup \{d(v,\Omega(\rho)\times \Rb) :v\in \Gamma(\alpha_n^+,\alpha_n^-,\Rb)\}\to 0$ as $n\to \infty$, then $r(\alpha_n) \to +\infty$ as $n \to \infty$.
\end{lemma}

    

\begin{proof}
    Suppose otherwise. Up to subsequence and conjugating $\alpha_n$, we may assume that a geodesic joining the fixed points of $\rho(\alpha_n)^+$ and $\rho(\alpha_n)^-$ intersects with a fixed compact subset of $X$. Then the pairs $(\xi_\rho(\alpha_n^+),\xi_\rho(\alpha_n^-))$ lie in a compact subset of $\partial_\infty^{(2)}X$. Up to further subsequence, suppose $(\alpha_n^+,\alpha_n^-)\to(a,b)$ as $n\to \infty$. Then $a\ne b$ and $\xi_\rho(a)\ne\xi_\rho(b)$. On the other hand, $(\alpha_n^+,\alpha_n^-,0)\to(a,b,0)$ in $\Gc(\Gamma)$, and our assumption implies that its distance to $\Omega(\rho)\times\Rb$ goes to $0$ as $n\to \infty$. Then $(a,b)\in\Omega(\rho)$, which gives a contradiction.
\end{proof}

We may also consider a stronger condition for the defining sequences of $\hEmesc(\rho)$ and let \[\begin{aligned}\hEpmesc(\rho) = \left\{ \mu\in \PCurr(\Gamma): \begin{array}{l}\text{There is a sequence of weighted multigeodesics} \\  \Y_n \text{ that strongly converges to } \mu \\ \text{ and with } \lim_{n \to \infty} r(\Y_n) = +\infty \end{array}\right\}~.\end{aligned}\] We again write $\Epmesc(\rho) = \bigcup_{\mu \in \hEpmesc(\rho)} \supp \mu$. It is clear that $\hEmesc'(\rho)\subset \hEmesc(\rho)$ hence $\Epmesc(\rho)\subset \Emesc(\rho)$. Since the strong convergence implies topological convergence, we have $\Epmesc(\rho) \subset \Ehausesc(\rho)$.

\begin{proposition}\label{proposition: ending currents contained in far away currents}
    Suppose $\xi_\rho$ is a Cannon--Thurston map of $\rho$. Then $\hEcont(\rho) \subset \hEpmesc(\rho)$.
\end{proposition}

\begin{proof}
    Suppose $\mu \in \hEcont(\rho)$. By Theorem~\ref{theorem: density of multigeodesics}, there exists a sequence of weighted multigeodesics $\Y_n = \sum_{j=1}^{k_n} a_{nj}\gamma_{nj}$ that strongly converges to $\mu$. In particular, \[\lim_{n \to \infty} \bigcup_{j=1}^{k_n}\Gamma(\gamma_{nj}^+,\gamma_{nj}^-,\Rb) = (\supp \mu) \times \Rb\subset \Gc(\Gamma)~.\]
    By Lemma~\ref{lemma: ending currents contained in Cannon--Thurston gluing}, we have that $\supp \mu \subset \Omega(\rho)$. Then by applying Lemma~\ref{lemma: sequences of geodesics converges into Cannon--Thurston gluing escapes}, for the sequence $(\gamma_{nj_n})_{n\in \Nb}$ where $r(\gamma_{nj_n}) = r(\Y_n)$, we have \[\lim_{n \to \infty} r(\Y_n) = \lim_{n \to \infty} r(\gamma_{nj_n}) = +\infty~.\]
\end{proof}

As a consequence, $\hEcont(\rho) \subset \hEpmesc(\rho) \subset \hEmesc(\rho)$ and together with Proposition~\ref{proposition: far away currents contained in ending currents} $\hEcont(\rho) = \hEpmesc(\rho) = \hEmesc(\rho)$ and $\Emesc(\rho) = \Econt(\rho) = \Epmesc(\rho) \subset \Ehausesc(\rho).$ Together with the following lemma, we complete the proof of Theorem~\ref{main theorem: the relations among ending currents}~(3) and (4).

The proof of Theorem~\ref{main theorem: the relations among ending currents}~(5) is similar to the argument above.

\begin{corollary}
    Suppose $\mu\in \PCurr(\Gamma)$ is such that $\supp \mu\subset \Omega(\rho)$, then $\mu \in \hEcont(\rho)$.
\end{corollary}
\begin{proof}
    Similar to the proof of Proposition~\ref{proposition: ending currents contained in far away currents}, let $\Y_n$ be a sequence of weighted multigeodesics that strongly converges to $\mu$. Then $\lim_{n \to \infty} r(\Y_n) = +\infty$ and hence $\mu \in \hEpmesc(\rho) =\hEcont(\rho)$.
\end{proof}

\begin{lemma}\label{proposition: topologically far away currents contained in Cannon--Thurston gluing}
    Suppose $\rho$ admits a Cannon--Thurston map $\xi_\rho$, then $\Ehausesc(\rho) \subset \Omega(\rho)$.
\end{lemma}

\begin{proof}
    Let $(x,y) \in \Ehausesc(\rho)$. By definition, there exists a sequence $(\gamma_n)_{n\in \Nb} \subset \Gamma$ such that $\lim_{n \to \infty}(\gamma_n^+,\gamma_n^-) = (x,y)$ and $\lim_{n \to \infty} r(\gamma_n) = +\infty$. By Section~\ref{section: visual metric}, this implies that the visual distance between $\rho(\gamma_n)^+ = \xi_\rho(\gamma_n^+)$ and $\rho(\gamma_n)^- =\xi_\rho(\gamma_n^-)$, the endpoints of the axis goes to $0$. Therefore by the continuity of $\xi_\rho$ we have \[\xi_\rho(x) = \lim_{n \to \infty}\xi_\rho(\gamma_n^+) = \lim_{n \to \infty}\xi_\rho(\gamma_n^-) = \xi_\rho(y)~.\]
\end{proof}

\subsection{Hyperbolic group extensions}\label{subsection: ending currents of hyperbolic group extensions}

Suppose $G$ is a hyperbolic group and $\Gamma$ is a normal subgroup of $G$ that is also hyperbolic. Mj \cite{mitra1998cannon} showed that in this setting there exists a Cannon--Thurston map $\partial_\infty \Gamma \to \partial_\infty G$. We assume $\Gamma = \pi_1(\Sigma)$ where $\Sigma$ is a closed hyperbolic surface. $G$ acts on $\pi_1(\Sigma)$ by conjugation hence induces a representation $G \to \Aut(\pi_1(\Sigma))$ with finite kernel as $G$ is hyperbolic. Hence up to finite index, $H = G/\pi_1(\Sigma)$ is a subgroup of the mapping class group $\Mod(\Sigma)$. On the other hand, any subgroup $H$ of $\Mod(\Sigma)$ gives rise to a group extension \[1 \longrightarrow \pi_1(\Sigma) \longrightarrow G \longrightarrow H \longrightarrow 1~.\] Farb--Mosher \cite{farb2002convex} introduced the notion of convex-cocompact subgroups of the mapping class group $\Mod(\Sigma)$. Together with the work of Hamenst{\"a}dt \cite{hamenstadt2005word}, $G$ is hyperbolic if and only if $H$ is a convex-cocompact subgroup of $\Mod(\Sigma)$. Such convex-cocompactness gives rise to a limit map \[\xi_H : \partial_\infty H \to \PML(\Sigma)~,\] where $\PML(\Sigma)$ denote the set of projective measured laminations which is the boundary of $\Teich(\Sigma)$. We view $\PML(\Sigma)$ as a subset of $\PCurr(\pi_1(\Sigma))$ via the canonical embedding of measured laminations into geodesic currents.

We suppose now that $H$ is a convex-cocompact subgroup of $\Mod(\Sigma)$ that gives rise to a hyperbolic group extension \[1 \longrightarrow \pi_1(\Sigma) \stackrel{\rho}{\longrightarrow} G \stackrel{p}{\longrightarrow} H \longrightarrow 1~.\] Following the notation of \cite{mj2018algebraic}, for a $\Gamma$-invariant subset $L \subset \partial_\infty^{(2)}\Gamma$, we define its \emph{diagonal closure} $L^d$ to be the equivalence relation on $\partial_\infty \Gamma$ generated by $L$.

\begin{theorem}[{\cite[Theorem~3.10, Theorem~3.12, Theorem~4.5]{mj2018algebraic}}]\label{theorem: diagonalized limit set of convex-cocompact representations}
    In the setting above, \[\Lambda(\rho) = \Omega(\rho) = \bigcup_{z \in \partial_\infty H} (\supp \xi_H(z))^d~.\]
\end{theorem}

By passing from the topological framework to our measure-theoretic one, we obtain a strict refinement of this result.

\begin{theorem}\label{theorem: limit set of convex-cocompact representations}
    In the setting above, \[\hEbdd(\rho) = \xi_H(\partial_\infty H)\quad \text{ and hence } \quad \Ebdd(\rho) = \bigcup_{z \in \partial_\infty H } \supp \xi_H(z)~.\]
\end{theorem}

\begin{proof}
    We first prove that $\xi_H(\partial_\infty H) \subset \hEbdd(\rho)$. For any $z\in \partial_\infty H$, let $(h_n)_{n\in \Nb}$ be a sequence of elements along a geodesic from $\id$ to $z$ in $H$ and for each $n\in \Nb$, let $\widetilde{h}_n$ be a preimage of $h_n$ in $G$ through the quotient $p$. We pick any $\alpha\in \pi_1(\Sigma)$, the homotopy class of an essential simple closed geodesic. Let $\alpha_n = \widetilde{h}_n \alpha \widetilde{h}_n^{-1}$. Then $\ell_\rho(\alpha_n) = \ell_\rho(\alpha)$ as the two elements are conjugated in $G$ and hence $(\ell_\rho(\alpha_n))_{n\in \Nb}$ is uniformly bounded. We claim that in $\pi_1(\Sigma)$, $\norm{\alpha_n}$ is divergent. Otherwise since there are only finitely many conjugacy classes of elements in $\pi_1(\Sigma)$ with $\norm{\cdot}$ uniformly bounded, we have up to subsequence, $\alpha_n$ is conjugated to $\alpha_m$ for any $m,n\in \Nb$. Then $h_m^{-1}h_n$ is a reducible mapping class in $H\subset \Mod(\Sigma)$, which contradicts that $H$ is convex-cocompact. Therefore we have $\widehat{\alpha}_n$ converges to an element in $\hEbdd(\rho)$. On the other hand $\widehat{\alpha}_n$ converges to $\xi_H(z)$ by the dynamics of $H$-action on the curve graph of $\Sigma$ (see \cite{kent2008shadows} and \cite{hamenstadt2005word}). Therefore $\xi_H(\partial_\infty H) \subset \hEbdd(\rho)$.

    We now prove the reverse inclusion. We fix a finite symmetric generating set $S$ of $G$ such that $S\cap \pi_1(\Sigma)$ generates $\pi_1(\Sigma)$ and $p(S)$ generates $H$. Suppose $\mu \in \hEbdd(\rho)$. By definition, there exists a sequence $(\gamma_n)_{n\in \Nb}\subset \pi_1(\Sigma)$ and a constant $L\geqslant 0$ such that $\lim_{n \to \infty} \widehat{\gamma}_n = \mu$, $\ell_\rho(\gamma_n)\leqslant L$ and $\lim_{n \to \infty} \norm{\gamma_n} = +\infty$. Notice that $\rho$ is the embedding of $\pi_1(\Sigma)$ into $G$ as a subgroup. Then $\ell_\rho(\gamma_n)$ is the stable word length of $\gamma_n$ in $G$. There are only finitely many $G$-conjugacy classes of elements in $G$ with stable word length bounded by $L$. Therefore up to subsequence, there exists $\widetilde{h}_n\in G$ such that $\gamma_n = \widetilde{h}_n \alpha \widetilde{h}_n^{-1}$ for a fixed $\alpha\in \pi_1(\Sigma)$. Let $h_n = p(\widetilde{h}_n) \in H$. We claim that $(h_n)_{n\in \Nb}$ is unbounded. Otherwise, $\widetilde{h}_n$ is contained in finitely many coset of $\pi_1(\Sigma)$ in $G$, then up to subsequence, elements of $(\gamma_n)_{n\in \Nb} = (\widetilde{h}_n \alpha \widetilde{h}_n^{-1})_{n\in \Nb}$ are $\pi_1(\Sigma)$-conjugated, which contradicts that $\lim_{n \to \infty} \norm{\gamma_n} = +\infty$. Therefore we may assume $h_n\to z\in \partial_\infty H$ as $n \to \infty$.

    For $x\in \Teich(\Sigma)$, we denote by $L_x\in \Curr(\pi_1(\Sigma))$ the Liouville current of $x$, and by $\ell_x:\pi_1(\Sigma) \to \Rb_{\geqslant 0}$ the length function for the hyperbolic metric $x$. Let $i(\cdot,\cdot)$ denote the intersection number function of geodesic currents.
    
    We now show that $\mu = \xi_H(z)$. Let $x_0\in \Teich(\Sigma)$ be a fixed base point. Then $\ell_{x_0}(\cdot)$ and $\norm{\cdot}$ are bilipschitz. Let $x_n = h_n x_0$. By \cite{farb2002convex}, the orbit map $\tau_{H,x_0}: H \to \Teich(\Sigma)$ defined by $h \mapsto h x_0$ extended continuously to \[\tau_{H,x_0}\cup \xi_H: H\cup \partial_\infty H \to \Teich(\Sigma) \cup \PML(\Sigma)~,\] hence $\lim_{n \to \infty}h_n x_0 = \xi_H(z)$. We denote by $\xi_H(z)_1 \in \Curr(\pi_1(\Sigma))$ the normalized geodesic current in the equivalence class of $\xi_H(z)\in \PCurr(\pi_1(\Sigma))$. Following Thurston's compactification, there exists a sequence of positive numbers $(c_n)_{n\in \Nb}$ such that \[\lim_{n \to \infty}c_n= 0 \quad \text{ and } \quad \lim_{n \to \infty} c_n L_{x_n} = \xi_H(z)_1~.\] Then for any $\gamma\in \pi_1(\Sigma)$ and $x\in \Teich(\Sigma)$, $\ell_x(\gamma) = i(L_x,\widehat{\gamma})$. Since $\gamma_n = \widetilde{h}_n \alpha \widetilde{h}_n^{-1}$, we have that the mapping class $h_n$ maps the $\pi_1(\Sigma)$-conjugacy class of $\alpha$ to the $\pi_1(\Sigma)$-conjugacy class $\gamma_n$, hence  $\ell_{x_n}(\gamma_n) = \ell_{h_n x_0}(\widetilde{h}_n \alpha \widetilde{h}_n^{-1}) = \ell_{x_0}(\alpha)$. Let $\mu_1\in \Curr(\pi_1(\Sigma))$ be the normalized geodesic current in the equivalence class of $\mu\in \PCurr(\pi_1(\Sigma))$. For now we view $\widehat{\gamma}_n \in \Curr(\pi_1(\Sigma))$ rather than $\PCurr(\Gamma)$ and denote that \[\nu_n = \frac{\widehat{\gamma}_n}{i(L_{x_0},\gamma_n)} = \frac{\widehat{\gamma}_n}{\ell_{x_0}(\gamma_n)}~.\] By $\lim_{n \to \infty}\norm{\gamma_n} = \lim_{n \to \infty} \ell_{x_0}(\gamma_n) = +\infty$ we get \[i(\nu_n, \nu_n) = \frac{i(\widehat{\gamma}_n,\widehat{\gamma}_n)}{\ell_{x_0}(\gamma_n)^2} = \frac{i(\widehat{\alpha},\widehat{\alpha})}{\ell_{x_0}(\gamma_n)^2} \to 0 \quad \text{ as }\quad n \to \infty~,\] \[\text{ and }\quad i(c_n L_{x_n}, \nu_n) = c_n\frac{\ell_{x_n}(\gamma_n)}{\ell_{x_0}(\gamma_n)} = c_n\frac{\ell_{x_0}(\alpha)}{\ell_{x_0}(\gamma_n)} \to 0 \quad \text{ as }\quad n \to \infty~.\] 
    Since  $\nu_n \to \mu_1$ as $n \to \infty$, by the continuity of $i$ following from \cite{bonahon1988geometry}, we have $i(\mu_1,\mu_1) = 0$, which indicates that $\mu_1$ is a measured lamination, and $i(\xi_H(z)_1,\mu_1) = 0$. Since $\xi_H(z)$ is filling and uniquely ergodic (\cite{farb2002convex,kent2008shadows}), then $\mu = \xi_H(z)$.
\end{proof}

\begin{example}\label{example: mapping torus}
    A cyclic subgroup in $\Mod(\Sigma)$ is convex-cocompact if and only if it is generated by a pseudo-Anosov mapping class $\phi$. In this case $\Ebdd(\rho) = \Econt(\rho) = \Eesc(\rho) = \Emesc(\rho)$ is the union of the attracting and repelling laminations of $\phi$ by Proposition~\ref{proposition: 3 dimensional ending current union of ending laminations}.
\end{example}

We follow the notations as above. When $H$ is non-elementary and hence $\partial_\infty H$ is a perfect set, we present an example such that $\Ebdd(\rho) \subsetneq \Econt(\rho)$.

\begin{example}\label{example: triangles limit to a polygon}
    We choose two pseudo-Anosov mapping classes $f,g\in \Mod(\Sigma)$ with the following properties (see \cite[Theorem~2.1]{lanneau2011minimum}, \cite{masur1993quadratic} for the existence),
    \begin{enumerate}
        \item the attracting lamination $\Lambda_f^+$ of $f$ is maximal, that is, every complementary region of $\Lambda_f^+$ is an ideal triangle;
        \item the attracting lamination $\Lambda_g^+$ of $g$ is not maximal, that is, some complementary region $P$ of $\Lambda_g^+$ is an ideal polygon with at least four sides.
    \end{enumerate}
    By \cite[Theorem~1.4]{farb2002convex}, after replacing $f$ and $g$ by sufficiently large positive powers, the subgroup $H=\langle f,g\rangle \subset \Mod(\Sigma)$ is free of rank $2$ and convex-cocompact. We still denote by \[1 \longrightarrow \pi_1(\Sigma) \stackrel{\rho}{\longrightarrow} G \stackrel{p}{\longrightarrow} H \longrightarrow 1\] the corresponding hyperbolic group extension. Let $z_0=g^+\in \partial_\infty H $ and $z_n=(g^nfg^{-n})^+=g^nf^+\in \partial_\infty H$. Then we have $z_n\to z_0$ as $n \to \infty$. Set \[\Lambda=\supp \xi_H(z_0)=\Lambda_g^+\quad \text{ and } \quad \Lambda_n=\supp \xi_H(z_n)=g^n(\Lambda_f^+)~.\] Each $\Lambda_n$ is maximal, while $\Lambda$ has a complementary ideal polygon with at least four sides. Then there is a diagonal leaf $\ell \subset \Lambda^d \setminus \Lambda$ that is contained in the Hausdorff limit $\lim_{n \to \infty}\Lambda_n$. Let $\alpha\in \pi_1(\Sigma)$ be the homotopy class of a fixed essential simple closed geodesic. 
    
    Let $\widetilde{f}$ and $\widetilde{g}$ be some $p$-preimage of $f$ and $g$ in $G$ respectively. For each $n$, let $\alpha_{n,m} = (\widetilde{g}^n \widetilde{f}^m) \alpha (\widetilde{g}^n \widetilde{f}^m )^{-1}$. Then by the same argument as in the proof of Theorem~\ref{theorem: limit set of convex-cocompact representations}, \[\lim_{m\to \infty}\widehat{\alpha}_{n,m} = \xi_H(z_n) \quad \text{ and } \quad \lim_{m \to \infty} \frac{\ell_\rho(\alpha_{n,m})}{\norm{\alpha_{n,m}}} = 0~.\] Then up to choose large enough $((m_{n,k})_{k=1}^n)_{n\in \Nb}\subset \Nb$, the sequence of multigeodesics $((\alpha_{k,m_{n,k}})_{k=1}^n)_{n\in \Nb}$ is such that \[\lim_{n \to \infty}\max_{1\leqslant k \leqslant n}\Big\{\frac{\ell_\rho(\alpha_{k,m_{n,k}})}{\norm{\alpha_{k,m_{n,k}}}}\Big\} = 0~.\] By further a proper choice of $(\{a_{nk}\}_{k=1}^{n})_{n\in \Nb}\subset \Rb_{>0}$, we have \[\lim_{n \to \infty}\sum_{k=1}^{n}a_{nk}\widehat{\alpha}_{k,m_{n,k}} = \sum_{n=1}^{\infty} \frac{1}{n^2}\xi_H(z_n)_1\] which projects to an element $\mu \in \hEcont(\rho)$. Here $\xi_H(z_n)_1 \in \Curr(\pi_1(\Sigma))$ is the normalized geodesic current in the equivalence class of $\xi_H(z_n)\in \PCurr(\pi_1(\Sigma))$. Therefore $\supp \mu = \overline{\bigcup \Lambda_n} \supset \ell$.

    Again by \cite{farb2002convex,kent2008shadows}, every point of $\xi_H(\partial_\infty H)$ is filling and uniquely ergodic. In particular, $\supp \xi_H(z)$ is a minimal lamination for each $z\in \partial_\infty H$. As $\ell$ is a leaf of $\Lambda^d$, the closure of the $\pi_1(\Sigma)$-orbit of $\ell$ contains $\Lambda$. If $\ell$ is a leaf of $\supp \xi_H(z')$ for some $z'\in \partial_\infty H \setminus\{z\}$, then the closure of the $\pi_1(\Sigma)$-orbit of $\ell$ is $\supp \xi_H(z')$. However, by \cite[Proposition~5.1]{mitra1997ending}, we have that for any distinct $y_1, y_2 \in \partial_\infty H$, $\supp \xi_H(y_1)$ and $\supp \xi_H(y_2)$ share no common leaves. Hence we have \[\ell \not\in \bigcup_{z \in \partial_\infty H } \supp \xi_H(z)~.\] Therefore together with Lemma~\ref{lemma: ending currents contained in Cannon--Thurston gluing}, Theorem~\ref{theorem: diagonalized limit set of convex-cocompact representations} and Theorem~\ref{theorem: limit set of convex-cocompact representations} we have that \[\bigcup_{z \in \partial_\infty H } \supp \xi_H(z) = \Ebdd(\rho) \subsetneq \Econt(\rho) \subset \Omega(\rho) = \bigcup_{z \in \partial_\infty H} (\supp \xi_H(z))^d~.\]
\end{example}

\subsection{Ending currents of geometrically finite actions }\label{subsection: ending currents of geometrically finite representations}

We still assume that $\Gamma$ is a hyperbolic group, $X$ is a $\delta$-hyperbolic space , $\rho:\Gamma \to \Isom(X)$ is a discrete representation with finite kernel, and $\Pc$ is a finite collection of finitely generated, infinite, quasiconvex subgroups of $\Gamma$. We also assume that $P\in\Pc$ is maximal in the sense that $P$ is the maximal group that preserves $\partial_\infty P$. Let \[\rho(\Pc) = \{\rho(P): P\in \Pc\}~.\] We discuss in this section how geometric finiteness is reflected by the ending currents.

Recall from the introduction that for any quasi-convex subgroup $P\subset \Gamma$, we denote $[P]=\Gamma\cdot\partial_\infty^{(2)}P$, and for any infinite-order element $\alpha\in\Gamma$, we denote $[\alpha]=[\langle\alpha\rangle]$.

\begin{theorem}\label{theorem: geometrically finite implies ending currents are parabolics}
    If $\rho:\Gamma \to \Isom(X)$ is geometrically finite with collection of parabolic subgroups $\Pc$, then \[\Ebdd(\rho) = \Eesc(\rho) =  \Emesc(\rho) = \Econt(\rho) = \Ehausesc(\rho) = \Omega(\rho) = \bigcup_{P\in \Pc} [P]~.\]
\end{theorem}

\begin{proof}
    Since $\rho:\Gamma \to \Isom(X)$ is geometrically finite, then any parabolic (with respect to $(\Gamma,\Pc)$) element $\gamma\in \Gamma$ has $\ell_\rho(\gamma) = 0$. Therefore, $\bigcup_{P\in \Pc} [P]\subset \Ebdd(\rho)$. By Theorem~\ref{theorem: Yaman criterion} and Theorem~\ref{theorem: conditions on relative hyperbolicity}, the Cannon--Thurston map of $\rho$ exists and exactly collapses $\gamma \cdot \partial_\infty^{(2)}P$ for each $\gamma\in \Gamma$ and $P\in \Pc$, hence $\Omega(\rho) = \bigcup_{P\in \Pc} [P]$. Therefore by Proposition~\ref{proposition: bounded ending currents contained in far away currents}, Proposition~\ref{proposition: far away currents contained in ending currents}, Proposition~\ref{proposition: ending currents contained in far away currents}, and Proposition~\ref{proposition: topologically far away currents contained in Cannon--Thurston gluing}, we have the desired equality.
\end{proof}

We directly deduce that

\begin{corollary}\label{corollay: example has ending current the self-intersecting geodesic}
    Let $\rho_0:\pi_1(\Sigma) \to  \Isom^+(\Hb^4)$ be the representation constructed in Section~\ref{section: construction of the example}, then \[\Ebdd(\rho_0) = \Eesc(\rho_0) =  \Emesc(\rho_0) = \Econt(\rho_0) = \Ehausesc(\rho_0) = \Omega(\rho_0) = [c]~.\]
\end{corollary}

\begin{lemma}\label{lemma: ending currents contained close curves are parabolic}
    Suppose $\rho$ admits a Cannon--Thurston map $\xi_\rho$. If one of $\Ebdd(\rho)$, $\Eesc(\rho)$ and $\Econt(\rho) = \Emesc(\rho)$ is identical to $\bigcup_{P\in \Pc} [P]$, then each $\rho(P)\in \rho(\Pc)$ is a parabolic subgroup, $\rho(P)$ is almost malnormal and hence $\Pc$ is almost malnormal.
\end{lemma}

\begin{proof}
    We first show that for each $P\in \Pc$, $\rho(P)$ is parabolic. By Lemma~\ref{lemma: sequences of geodesics converges into Cannon--Thurston gluing escapes} and Proposition~\ref{proposition: ending currents contained in far away currents} we have that $(\gamma^+,\gamma^-)\in \partial_\infty^{(2)}P\subset \Econt(\rho) = \Emesc(\rho) \subset \Omega(\rho)$ for any one of these cases. Then by Lemma~\ref{lemma: sequences of geodesics converges into Cannon--Thurston gluing escapes}, $r(\gamma) = +\infty$ and therefore $\rho(\gamma)$ is parabolic. Suppose $\rho(\gamma)$ fixes $p\in\partial_\infty X$ and there is another $\eta\in P$ such that $\rho(\eta)$ fixes $q\in\partial_\infty X$. If $p\ne q$, then for a large enough $N\in \Nb$, $\rho(\gamma^N\eta^N)$ is loxodromic by ping-pong lemma, which contradicts $\rho(\gamma^N\eta^N)$ being parabolic. Therefore all elements of $\rho(P)$ fix a common parabolic point and hence $\rho(P)$ is a parabolic subgroup.\\
    
    We then show that different parabolic subgroups fix different parabolic points. For any $P,P'\in \Pc$ and $\gamma\in\Gamma$, the fixed points of $\rho(P)$ and $\rho(\gamma P' \gamma^{-1})$ are the same only when $P=P'$ and $\gamma\in P$. Otherwise, suppose $\rho(P)$ and $\rho(\gamma P' \gamma^{-1})$ are distinct and both fix $p\in\partial_\infty X$. Let  $H = \ip{\rho(P), \ \rho(\gamma P' \gamma^{-1})}$ be the group generated by $\rho(P)$ and $\rho(\gamma P' \gamma^{-1})$. Then
    \begin{enumerate}
        \item either $H$ is parabolic. In this case $H$ is a strictly larger parabolic subgroup fixing $p$. In particular $\partial_\infty^{(2)}\langle P,\, \gamma P'\gamma^{-1}\rangle \subset \Emesc(\rho)$, contradicting our assumption on $\Emesc(\rho)$.
        \item or $H$ contains a loxodromic element $g$ fixing $p$. Pick $u\in \rho(P)\setminus \{\id\}$, then $g^n u g^{-n}\to \id$ as $n \to \infty$ (see for example \cite{tukia1994convergence}), contradicting the discreteness of $\rho$.
    \end{enumerate}
    In particular, $\Pc$ is almost malnormal.\\
\end{proof}

For Theorem~\ref{main theorem: geometric finiteness and ending currents}, the implications \eqref{item: contracting measure} $\Rightarrow$ \eqref{item: single far away} $\Rightarrow$ \eqref{item: bounded measure} follows immediately from Lemma~\ref{lemma: ending currents contained close curves are parabolic}. \eqref{item: yaman} $\Rightarrow$ \eqref{item: geometrically finite} follows from Theorem~\ref{theorem: Yaman criterion}. We now show \eqref{item: topological ending} $\Rightarrow$ \eqref{item: geometrically finite}.

\begin{theorem}
    Suppose $\rho$ admits a Cannon--Thurston map $\xi_\rho$. Then $\rho$ is geometrically finite with collection of parabolic subgroups $\Pc$ if and only if \[\Ehausesc(\rho) = \bigcup_{P\in \Pc} [P]~.\]
\end{theorem}

\begin{proof}
    Since $\Ehausesc(\rho)= \bigcup_{P\in \Pc} [P] \subset \Omega(\rho)$ by Proposition~\ref{proposition: topologically far away currents contained in Cannon--Thurston gluing}, we have that for each $P\in \Pc$, $\xi_{\rho\vert_P}:\partial_\infty P \to \partial_\infty X$ is a constant map. Then $\rho(P)$ has the unique fixed point the singleton $\xi_\rho(\partial_\infty P)$ and hence $\rho(P)$ is a parabolic subgroup. Therefore by Theorem~\ref{theorem: geometrically finite implies ending currents are parabolics}, we have $\Emesc(\rho\vert_P) = \Omega(\rho\vert_P) = \partial_\infty^{(2)}P \subset \Emesc(\rho)$, where the last inclusion follows the definition of $\Emesc(\cdot)$. On the other hand by Proposition~\ref{proposition: ending currents contained in far away currents} we have $\Emesc(\rho)\subset \Ehausesc(\rho)$ and hence \[\Emesc(\rho) = \Ehausesc(\rho) = \bigcup_{P\in \Pc} [P]~.\] By Lemma~\ref{lemma: ending currents contained close curves are parabolic} and Theorem~\ref{theorem: conditions on relative hyperbolicity}, we have $(\Gamma,\Pc)$ is relatively hyperbolic, then the theorem follows from Theorem~\ref{theorem: escaping loxodromic implies geometrically infinite} below.
\end{proof}

\begin{theorem}[{\cite{bonahon1986bouts},\cite{kapovich2018geometric},\cite[Theorem~1.1]{yang2026two}}]\label{theorem: escaping loxodromic implies geometrically infinite}
    Suppose $(\Gamma,\Pc)$ is a relatively hyperbolic group. Then $\rho$ is geometrically infinite if and only if there exists a sequence of loxodromic (with respect to $(\Gamma,\Pc)$) elements $(\gamma_n)_{n\in \Nb}$ such that $\lim_{n \to \infty}r(\gamma_n) = +\infty$.
\end{theorem}

They are standard facts that a hyperbolic group contains no abelian subgroup of rank $2$, and a geometrically finite Kleinian group in $\Isom^+(\Hb^n)$ has every parabolic subgroup virtually abelian of rank at most $n-1$. It follows directly that

\begin{corollary}
    Suppose $\rho:\Gamma \to \Isom^+(\Hb^n)$ is a discrete representation with finite kernel that admits a Cannon--Thurston map. Then $\rho$ is geometrically finite if and only if $\Ehausesc(\rho)$ consists of finitely many $\Gamma$-orbits of closed geodesics in $\Gamma$.
\end{corollary}

While $\Ehausesc(\rho)$ characterizes relative geometric finiteness cleanly, things become significantly more subtle when working with $\Ebdd(\rho)$ and $\Econt(\rho)$. The following characterization for convex-cocompact case is straightforward.

\begin{proposition}
    The representation $\rho:\Gamma \to \Isom(X)$ is convex-cocompact if and only if $\Econt(\rho)$ is empty.
\end{proposition}

\begin{proof}
    The forward direction is trivial, and the backward direction easily follows from that when $\Econt(\rho)$ is empty, $\frac{\ell_\rho(\gamma)}{\norm{\gamma}}$ has a positive uniform lower bound for any $\gamma\in \Gamma$.
\end{proof}

However, when discuss the characterization of geometrically finite representations, although Lemma~\ref{lemma: ending currents contained close curves are parabolic} gives strong control over the structure of $\rho(\Pc)$, it does not by itself rule out a ``hidden geodesic with vanishing mass'' phenomenon on the level of measures. Actually, there exist geometrically infinite representations $\rho$ and a sequence of wandering closed geodesics, represented by loxodromic elements of $\Gamma$, that escape to infinity and therefore contribute a point to $\Omega(\rho)$, while their normalized counting measures concentrate almost entirely on the parabolic subgroups, that is, these geodesics may spend almost all of their word length deep inside horoballs. In that case, the limiting measure in $\Econt(\rho)$ could still be supported purely on the parabolic part, and fail to detect the collapse of the thick part. We construct such an example in the section below.

\subsection{A geometrically infinite example}\label{section: geometrically infinite example}

Our construction is motivated by \cite{bigdely2012non} and \cite{chong2024embedding} that studies properties of ascending HNN extensions.

Let $\Gamma$ be a free group of two generators $\alpha,\beta$ and let $P$ be the cyclic subgroup of $\Gamma$ generated by $\alpha$. We fix $m\in \Nb$ large enough and let $\phi: \Gamma\to \Gamma$ be the injective endomorphism defined by $\phi(\alpha) = \alpha^m$ and $\phi(\beta) =\beta \alpha\beta \alpha^2\dots \beta\alpha^{m-1}\beta$. We consider the group \[G = \langle \alpha,\beta,t:t\alpha t^{-1}=\phi(\alpha),\ t\beta t^{-1}=\phi(\beta) \rangle~.\] Then the subgroup of $G$ generated by $\alpha$ and $t$ is an embedding of 
\[Q = \langle\alpha,t:t\alpha t^{-1}=\phi(\alpha)\rangle\] 
by Theorem~\ref{theorem: 1/6 embedding factor}. Note that $Q\cong \mathrm{BS}(1,m)$, the so-called Baumslag--Solitar group (\cite{baumslag1962some}), and in particular both $\alpha$ and $t$ are of infinite order. One can verify that $G$ is hyperbolic relative to $Q$ by apply Theorem~\ref{theorem: 1/6 relatively hyperbolic} to \[G = (\langle \beta\rangle * Q)/ \langle \langle t\beta^{-1} t^{-1}\beta \alpha\beta \alpha^2\dots \beta\alpha^{m-1}\beta \rangle \rangle~.\]
Then let $X = \Tc(G,Q)^h$ be the Gromov model of $(G,Q)$ constructed in Proposition~\ref{proposition: existence of Cannon--Thurston map for the example}. The representation $\rho: \Gamma \to G$ defined by $\rho(\alpha)=\alpha$ and $\rho(\beta)=\beta$ induces an action of $\Gamma$ on $X$. We show that $\rho$ admits a Cannon--Thurston map $\partial_\infty(\Gamma,P) \to \partial_\infty X$ also in Proposition~\ref{proposition: existence of Cannon--Thurston map for the example}. By composing the projection $\partial_\infty \Gamma \to \partial_\infty(\Gamma,\Pc)$ provided by Theorem~\ref{theorem: conditions on relative hyperbolicity}, we get a Cannon--Thurston map $\xi: \partial_\infty\Gamma \to \partial_\infty X$. We show that $\hEcont = \{\widehat \alpha\}$ and $\Econt(\rho) \subsetneq \Omega(\rho)$ by the following propositions.

\begin{proposition}\label{proposition: econt of geometrically finite counterexample}
    For the representation $\rho$ above, $\hEcont = \{\widehat \alpha\}$ and hence $\Econt = [\alpha]$.
\end{proposition}

\begin{proposition}\label{proposition: econt not Omega}
    Let $x = \lim_{n \to \infty} \phi^n(\beta)\in \partial_\infty \Gamma$. Then $x\ne \alpha^\pm$ but $\xi(x) = \xi(\alpha^+)$. In particular $(x,\alpha^+)\in \Omega(\rho)\setminus \Econt(\rho)$.
\end{proposition}

\begin{proof}[Proof of Proposition~\ref{proposition: econt not Omega}]
    Notice that $\phi^n(\beta)$ is always a word starting with the letter $\beta$. Then $x = \lim_{n \to \infty} \phi^n(\beta) \ne \alpha^\pm$.
    Notice also that $x$ is a fixed point of $t$ by the computation
    \[\xi(x)= \lim_{n\to\infty} \rho(\phi^n(\beta))=\lim_{n\to\infty} \rho(\phi^{n+1}(\beta)) = \lim_{n\to\infty} t \rho(\phi^n(\beta))t^{-1}=t\xi(x)~.\] There is a unique parabolic point $q$ in $\partial_\infty(G,Q)$ that is fixed by $Q$. Since $\alpha,t\in Q$ are of infinite order, $q$ is the unique fixed point of them. Therefore $\xi(x) = q = \xi(\alpha^+)$.
\end{proof}


Now we go back to the proof of Proposition~\ref{proposition: econt of geometrically finite counterexample}.
For any $\gamma\in \Gamma$, we denote by $A(\gamma)$ (respectively $B(\gamma)$) the number of $\alpha^{\pm 1}$ (respectively $\beta^{\pm 1}$) appeared in the cyclically reduced conjugation of $\gamma$ written as a reduced word in letters $\alpha^{\pm 1},\beta^{\pm 1}$.

It is easy to check by direct computation that for any $\gamma\in \Gamma$ and $n\in \Nb$,
\begin{equation}\label{equation: B}
    B(\phi^n(\gamma)) = m^n B(\gamma)~,
\end{equation}
and
\begin{equation}\label{equation: A}
    A(\phi^n(\gamma)) = m^n( A(\gamma)+ \frac{n(m-1)}{2} B(\gamma))~.
\end{equation}

Therefore, for any fixed $\gamma\in \Gamma\setminus\{\id\}$ with $B(\gamma)\ne 0$, we have \[\frac{A(\phi^n(\gamma))}{B(\phi^n(\gamma))} = \frac{A(\gamma)}{B(\gamma)} + \frac{n(m-1)}{2} \to +\infty \quad \text{as} \quad n\to \infty~,\] and hence the proportion of letters $\beta^{\pm 1}$ in the word $\phi^n(\gamma)$ goes to $0$ as $n\to \infty$.
Then the sequence of counting measures on $\phi^n(\gamma)$ converges to $\widehat \alpha$. Note that $\ell_\rho(\phi^n(\gamma)) = \ell_\rho(\gamma)$ and hence $\frac{\ell_\rho(\phi^n(\gamma))}{\norm{\phi^n(\gamma)}}\to 0$ as $n\to \infty$.
To prove the proposition, we need to show that for any sequence $(\gamma_n)$ such that $\lim_{n\to\infty}\frac{\ell_\rho(\gamma_n)}{\norm{\gamma_n}}= 0$, their counting measures converges to $\widehat \alpha$, rather than only the sequence $(\phi^n(\gamma))$ for a fixed $\gamma$.

\begin{lemma}\label{lemma: mystery computation}
    For any $r\in\Nb$ and $\gamma\in\Gamma$, \[\left(1+\frac{r(m-1)}{2}\right)B(\gamma) \leqslant \norm{\gamma}+m^{r+1}\ell_\rho(\gamma)~.\]
\end{lemma}

\begin{proof}
    For convenience, we write $\omega_n=\phi^n(\beta)$ and $c=(m-1)/2$. For every integer $r\geqslant 0$ and $\gamma\in\Gamma$, we set \[F_r(\gamma)=crB(\gamma)-A(\gamma)~.\] Whenever $\gamma,\eta$ are cyclically reduced words whose concatenation $\gamma\eta$ is also in cyclically reduced form, we have
    \begin{equation}\label{equation: concatenation}
        F_r(\gamma\eta)=F_r(\gamma)+F_r(\eta).
    \end{equation}

    We first show that, for every $n,r\geqslant 0$, if $u$ is a subword of $\omega_n$,
    \begin{equation}\label{equation: subword}
        F_r(u)\leqslant m(m^r-1).
    \end{equation} and if $u$ is an initial segment or a terminal segment of $\omega_r$,
    \begin{equation}\label{equation: initial terminal}
        F_r(u)\leqslant \frac{m}{2}(m^r-1)
    \end{equation}

    We proceed by induction on $r$. It is clear that when $r=0$, \eqref{equation: subword} and \eqref{equation: initial terminal} hold. Assume \eqref{equation: subword} and \eqref{equation: initial terminal} holds for $r-1$ and every $n\in \Nb$.

    \noindent{\bf Case 1.} $n\leqslant r$. By the induction hypothesis and $B(u)\leqslant B(\omega_n) =m^n\leqslant m^r$, 
    \[F_r(u) = F_{r-1}(u)+cB(u) \leqslant m(m^{r-1}-1)+cm^r=\frac{m+1}{2}m^r-m \leqslant m(m^r-1)~.\] This proves \eqref{equation: subword} in this case.

    For \eqref{equation: initial terminal}, suppose $u$ is an initial segment of $\omega_r$. Notice that \[\omega_r = \omega_{r-1} \alpha^{m^{r-1}} \omega_{r-1} \alpha^{2m^{r-1}} \dots \omega_{r-1} \alpha^{(m-1)m^{r-1}} \omega_{r-1}~.\] 
    Then $u$ is a concatenation of several complete copies of $\omega_{r-1}$, powers of $\alpha$, and at most one copy $v$ of an initial segment of $\omega_{r-1}$. Since $F_{r-1}(\omega_{r-1})=0$ and $F_{r-1}(\alpha^j)\leqslant 0$ for $j\geqslant 0$, by \eqref{equation: concatenation}, we have \[F_r(u) \leqslant F_{r-1}(u) + cB(u) \leqslant F_{r-1}(v) + cB(u)\leqslant \frac{m}{2}(m^{r-1}-1) +cm^r \leqslant \frac{m}{2}(m^r-1)~.\] Similar argument applies to the case when $u$ is a terminal segment of $\omega_r$.

    \noindent{\bf Case 2.} $n>r$. Since $\omega_n=\phi^{n-r}(\omega_r)$, the word $\omega_n$ is a concatenation of copies of $\omega_r$ and powers of $\alpha$. If $u$ is a subword in a single copy of $\omega_r$, the desired estimate follows from Case 2. Otherwise, $u$ consists of complete copies of $\omega_r$, powers of $\alpha$, and at most two subwords $v,v'$ at the beginning or the end that are initial or terminal segment of $\omega_r$. Again since $F_r(\omega_r)=0$ and $F_r(\alpha^j)\leqslant 0$ for $j\geqslant 0$, by \eqref{equation: concatenation} and \eqref{equation: initial terminal},
    we have \[F_r(u) \leqslant F_r(v)+F_r(v') \leqslant \frac{m}{2}(m^r-1)+ \frac{m}{2}(m^r-1) = m(m^r-1)~.\]
    We completed the proof for \eqref{equation: subword} and \eqref{equation: initial terminal}.

    We now prove the inequality \[(1+cr)B(\gamma) \leqslant \norm{\gamma}+m^{r+1}\ell_\rho(\gamma)~.\]

    Since $A(\gamma)$, $B(\gamma)$, $\norm{\gamma}$ and $\ell_\rho(\gamma)$ are all $\Gamma$-conjugation invariant, we may assume that $\gamma$ is cyclically reduced. We choose a geodesic from $\id$ to $\gamma$ in $X$. Tracking all the vertices in $\Gamma$ of this geodesic determines an expression \[\gamma=q_0\beta^{\epsilon_1}q_1\dots \beta^{\epsilon_k}q_k, \quad q_i\in Q,\quad \epsilon_i=\pm 1~,\] where $k$ is the number of $\beta^{\pm1}$ in this expression. Then $k\leqslant d_X(\id,\gamma)$. By looking at which coset of $Q/P$ contains $q_i$, we can find $h_i\in \Zb$ such that $q_i t^{-h_i}$ is a power of $\alpha$. Since $\gamma\in \Gamma$, we have $h_1 + h_2 + \dots + h_k = 0$. Therefore, for a large enough $N\in \Nb$,
    \[\begin{aligned} 
        \phi^N(\gamma) = &t^N( q_0\beta^{\epsilon_1}q_1\dots \beta^{\epsilon_k}q_k)t^{-N} \\
        = &(t^N q_0 t^{-N-h_0}) (t^{N+h_0} \beta^{\epsilon_1}  t^{-N-h_0}) (t^{N+h_0} q_1 t^{-N-h_0-h_1}) \dots \\ 
        & (t^{N+h_0+\dots +h_{k-1}} \beta^{\epsilon_k}t^{-N-h_0-\dots -h_{k-1}}) (t^{N+h_0+\dots +h_{k-1}} q_k t^{-N-h_0-\dots -h_k}) \\
        = & \alpha^{s_0}\omega_{n_1}^{\epsilon_1}\alpha^{s_1} \dots \omega_{n_k}^{\epsilon_k}\alpha^{s_k}
    \end{aligned}\] for some $s_0,s_1,\dots, s_k \in \Nb$.
    Note that some $q_i$ might be $\id$, hence $\phi^N(\gamma)$ is a concatenation of finitely many powers $\alpha$ and at most $k$ many subwords of $\omega_{n_i}$. By \eqref{equation: concatenation} and \eqref{equation: subword},
    \[F_{N+r}(\phi^N(\gamma) )\leqslant k m (m^{N+r}-1)~.\]
    On the other hand, by direct computation using \eqref{equation: B} and \eqref{equation: A}, we have \[F_{N+r}(\phi^N(\gamma) ) = m^N F_r(\gamma)~.\]
    Therefore, \[cr B(\gamma) -A(\gamma) = F_r(\gamma) \leqslant k m^{1-N} (m^{N+r}-1) \leqslant k m^{r+1}\leqslant m^{r+1} d_X(\id,\gamma)~.\]
    Note also that $\abs{\gamma} = A(\gamma) + B(\gamma)$, we conclude \[(1+cr)B(\gamma) \leqslant \abs{\gamma}+m^{r+1}d_X(\id,\gamma)~.\] By applying this inequality to $\gamma^n$, we get \[(1+cr)B(\gamma) = (1+cr)\frac{B(\gamma^n)}{n}\leqslant \frac{\abs{\gamma^n}}{n}+m^{r+1}\frac{d_X(\id,\gamma^n)}{n}~.\] Finally we complete the proof by taking $n\to \infty$.
\end{proof}

\begin{proof}[Proof of Proposition~\ref{proposition: econt of geometrically finite counterexample}]
    Let $(\gamma_n)\subset \Gamma$ be a sequence such that $\frac{\ell_\rho(\gamma_n)}{\norm{\gamma_n}}\to 0$ as $n \to \infty$. Then by Lemma~\ref{lemma: mystery computation}, \[\frac{B(\gamma_n)}{\norm{\gamma_n}} \leqslant \frac{2}{r(m+1)} + \frac{2m^{r+1}}{m+1} \frac{\ell_\rho(\gamma)}{\norm{\gamma_n}}\] for any $r\in \Nb$ and $n\in \Nb$. In particular for each fixed $r\in \Nb$, \[\limsup_{n\to\infty} \frac{B(\gamma_n)}{\norm{\gamma_n}} \leqslant \frac{2}{r(m+1)}~.\] Therefore $\frac{B(\gamma_n)}{\norm{\gamma_n}} \to 0$ and hence $\widehat \gamma_n \to \widehat \alpha$ as $n\to \infty$.
\end{proof}

\begin{remark}
    Note that, since a rank 1 Lie group cannot contain a discrete copy of $\mathrm{BS}(1,m)$, this construction does not occur in the hyperbolic setting. It is an interesting question whether, for group actions on $\Hb^n$, we have the implication \eqref{item: contracting measure} $\Rightarrow$ \eqref{item: geometrically finite} in \Cref{main theorem: geometric finiteness and ending currents}.
\end{remark}
\medskip

\appendix

\section{Relatively hyperbolic groups}\label{section: relatively hyperbolic groups}

\subsection{Relatively hyperbolic groups}

The concept of relative hyperbolicity for finitely generated groups was introduced by Gromov \cite{Gromov}. We briefly recall the definition and several standard properties that will be used throughout the paper.

\begin{definition}\label{definition: relatively hyperbolic groups}
    Let $\Gamma$ be a finitely generated group and let $\Pc$ be a finite collection of finitely generated, infinite subgroups of $\Gamma$. Denote by \[\Pc^\Gamma := \{\gamma P \gamma^{-1} \mid \gamma \in \Gamma,\; P \in \Pc\}\] the collection of all conjugates of subgroups in $\Pc$. The pair $(\Gamma,\Pc)$ is called a \emph{relatively hyperbolic group}, if $\Gamma$ acts properly discontinuously and by isometries on a $\delta$-hyperbolic space $(Y,d_Y)$ such that,
    \begin{enumerate}
        \item There exists a $\Gamma$-invariant collection \[\Bc = \{B_P \mid P \in \Pc^\Gamma\}\] of pairwise disjoint open horoballs in $Y$, where each $B_P$ is stabilized by $P$;
        \item The induced action of $\Gamma$ on the complement \[Y^{\mathrm{th}} := Y \setminus \bigcup_{B\in\Bc} B\] is cocompact.
    \end{enumerate}
    We call such a space $Y$ a \emph{Gromov model} of $(\Gamma,\Pc)$. We also say in this case $\Gamma$ is \emph{hyperbolic relative to }$\Pc$ or $\Pc^\Gamma$ when $(\Gamma,\Pc)$ is a relative hyperbolic group.
\end{definition}

Bowditch \cite{Bow12} shows that if $(\Gamma,\Pc)$ is a relative hyperbolic group, the Gromov boundary of any two Gromov models of $(\Gamma,\Pc)$ are $\Gamma$-equivariantly homeomorphic. We denote it by $\partial_\infty(\Gamma,\Pc)$
is called the \emph{Bowditch boundary} of $(\Gamma,P)$,
denoted by $\partial_\infty(\Gamma,\Pc)$,

Bowditch \cite{Bow12} shows that for a relatively hyperbolic group $(\Gamma,\Pc)$, the Gromov boundaries of any two Gromov models of $(\Gamma,\Pc)$ are $\Gamma$-equivariantly homeomorphic. We call this boundary, well defined up to $\Gamma$-equivariant homeomorphism, the \emph{Bowditch boundary} of $(\Gamma,\Pc)$ and denote it by $\partial_\infty(\Gamma,\Pc)$.

Moreover, the $\Gamma$-action on $\partial_\infty(\Gamma,\Pc)$ is \emph{geometrically finite} (in the sense of convergence group actions), that is, $\partial_\infty(\Gamma,\Pc)$ consists only of conical limit points and bounded parabolic points for the $\Gamma$-action. A theorem of Yaman \cite{Yaman} below tells that the $\Gamma$-action being geometrically finite (in the sense of convergence group actions) is a criterion for relative hyperbolicity.

\begin{theorem}[\cite{Yaman}]\label{theorem: Yaman criterion}
    Suppose a finitely generated group $\Gamma$ acts on a nonempty perfect metrizable compact set $M$ as a geometrically finite convergence group and there are only finitely many $\Gamma$-orbits of bounded parabolic points. Then $\Gamma$ is hyperbolic relative to the collection of stabilizers of these bounded parabolic points and its Bowditch boundary is $\Gamma$-equivariantly homeomorphic to $M$.
\end{theorem}

\begin{definition}\label{definition: geometrically finite representations}
    Suppose $\Gamma$ is a finitely generated group and $X$ is a $\delta$-hyperbolic space. A representation $\rho:\Gamma \to \Isom(X)$ is \emph{geometrically finite} if the $\Gamma$-action on its limit set $\Lambda_\Gamma \subset \partial_\infty X$ is geometrically finite (in the sense of convergence group actions).
\end{definition}

Equivalently, $\rho:\Gamma \to \Isom(X)$ is geometrically finite if the convex hull of $\Gamma \cdot o \subset X$ for a (any) base point $o\in X$ is a Gromov model for $(\Gamma,\Pc)$, where $\Pc$ is the collection of the stabilizers of the bounded parabolic points in $\Lambda_\Gamma$. In the case when $X=\Hb^n$, this definition agrees with the notion of geometric finiteness in the usual sense (see for instance \cite{bowditch1993geometrical} for these equivalent characterizations).

\subsection{Relative hyperbolization of hyperbolic groups}

When $\Gamma$ itself is a hyperbolic group, relative hyperbolicity can be characterized in terms of the geometry of subgroups. We say that a finite collection $\Pc$ of subgroups of $\Gamma$ is \emph{almost malnormal} if for any $P,P'\in \Pc$ and $\gamma\in\Gamma$, $P \cap \gamma P' \gamma^{-1}$ is finite unless $P=P'$ and $\gamma\in P$.

\begin{theorem}[\cite{Bow12}, \cite{Tran13}, \cite{Man15}]\label{theorem: conditions on relative hyperbolicity}
    Let $\Gamma$ be a hyperbolic group and let $\Pc$ be a finite collection of finitely generated, infinite subgroups of $\Gamma$. Then $(\Gamma,\Pc)$ is relatively hyperbolic if and only if $\Pc$ is almost malnormal and each subgroup $P \in \Pc$ is quasiconvex in $\Gamma$. Moreover, under these conditions, there exists a continuous $\Gamma$-equivariant quotient map \[\tau : \partial_\infty \Gamma \to \partial_\infty(\Gamma,\Pc)\] that exactly collapses the limit set of each $P \in \Pc^\Gamma$ to the parabolic point fixed by $P$.
\end{theorem}

\subsection{$C'(1/6)$-Criteria for relative hyperbolicity}\label{section: 1/6}

A good reference for the following criteria is \cite[Section~2]{martin2017combination}.

Suppose $F= A * B$ is a free product of two groups $A$ and $B$. We say $\gamma = q_1q_2\dots q_n$ is in the normal form if each $q_i$ is a non-trivial element in either $A$ or $B$ and no two consecutive $q_i$ and $q_{i+1}$ belong to the same free factor. We denote by $\abs{\gamma}_{\mathrm{A*B}} = n$ the length of the normal form.

We say a normal form $\gamma=q_1q_2\dots q_n$ is \emph{weakly cyclically reduced} if $\gamma =\id$ or $q_nq_1\ne \id$. We say $\gamma\eta$, a product of two normal form $\gamma=q_1q_2\dots q_n$ and $\eta = q_1'q_2'\dots q_m'$, is \emph{weakly reduced} if $q_n q_1' \ne \id$.

For a subset $R\subset F$ with elements represented by weakly cyclically reduced normal forms, let $G = F/ \langle \langle R \rangle \rangle$ where $\langle \langle R \rangle \rangle$ is the normal closure of $R$ in $F$. We assume $R$ is \emph{symmetrized} in the sense that if it is invartiant under taking weakly cyclically reduced conjugates and inverses.

An element $p\in F$ is a \emph{piece} of $R$ if there exists distinct $\gamma_1,\gamma_2\in R$ such that $\gamma_1 = p u_1$ and $\gamma_2 = p u_2$ are weakly reduced.

We say $R$ is $C'(1/6)$ if any piece $A$ and a relation $r=pu$ that is in weakly reduced satisfies $\abs{p}_{\mathrm{A*B}} < \frac{1}{6} \abs{r}_{\mathrm{A*B}}$.

\begin{theorem}[\cite{lyndoncombinatorial}]\label{theorem: 1/6 embedding factor}
    In the setting above, if $R$ is $C'(1/6)$, then $A$ and $B$ embed into $G$
\end{theorem}

\begin{theorem}[\cite{pankratiev1999hyperbolic}]\label{theorem: 1/6 relatively hyperbolic}
    In the setting above, if $R$ is $C'(1/6)$, then $G$ is hyperbolic relative to $\{A,B\}$. In particular, if $B$ is hyperbolic, then $(G,A)$ is relatively hyperbolic.
\end{theorem}

\subsection{Cannon--Thurston maps on Bowditch boundaries}\label{section: mj-pal}

We recall the following definitions from \cite[Section~1]{mj2011relative}.

Let $(Y,d_Y)$ be a geodesic metric space and let $\Hc$ be a collection of closed subsets whose path metrics are geodesic (that is, any two points are joined by a geodesic with respect to the path metric). Assume that the sets in $\Hc$ are uniformly separated, that is, there is $\epsilon>0$ such that $d_Y(H,H')\geqslant\epsilon$ for any two distinct $H,H'\in \Hc$. For each $H\in\Hc$, attach $H\times[0,\infty)$ to $H\subset Y$ along $H\times\{0\}$, giving vertical paths their usual length and horizontal paths at height $s$ their intrinsic length in $H$ multiplied by $2^{-s}$. Denote the resulting path metric space by $Y^h$. We say the pair $(Y,\Hc)$ is a \emph{relatively hyperbolic space}, or $Y$ is \emph{hyperbolic relative to} $\Hc$, if $Y^h$ is Gromov hyperbolic and we call the sets in $\Hc$ horosphere-like sets.

The \emph{coned-off space} $\widehat Y$ of $(Y,\Hc)$ is obtained from $Y$ by adding a point $c_H$ for each $H\in\Hc$ and joining every point of $H$ to $c_H$ by an edge of length $1/2$. We equip $\widehat Y$ with the resulting path metric.

\begin{definition}
    A \emph{tree of geodesic metric spaces} is a geodesic metric space $(\Xc,d_{\Xc})$ with a continuous projection $\mathrm{Pr}:\Xc\to T$ onto a simplicial tree $T$ with the following properties:
\begin{enumerate}
    \item For each vertex $v$ of $T$, the space $\Xc_v=\mathrm{Pr}^{-1}(v)$ is geodesic for the path metric $d_v$ on $\Xc_v$ induced by $d_{\Xc}$. 
    \item The inclusion of $X_v$ into $\Xc$ is uniformly proper, that is, for every $R \geqslant 0$, there is $M(R)<\infty$, such that for any vertex $v$ of $T$, if $x,y\in \Xc_v$ are such that $d_{\Xc}(x,y)\leqslant R$, then $ d_v(x,y)\leqslant M(R)$.
    \item For each edge $e$ of $T$ identified with the interval $[0,1]$, the space $\Xc_e=\mathrm{Pr}^{-1}(1/2)$ is geodesic for the path metric $d_e$ on $\Xc_e$ induced by $d_{\Xc}$. 
    \item There is a continuous {\em edge map} \[f_e:\Xc_e\times[0,1]\to \Xc\] satisfying $\mathrm{Pr}(f_e(x,s))=s\in e$, whose restriction to $\Xc_e\times(0,1)$, equipped with the product metric, is an isometry onto $\mathrm{Pr}^{-1}(0,1)$ with its path metric induced by $d_{\Xc}$. If $v$ is an endpoint of $e$, denote by $f_{e,v}:\Xc_e\to \Xc_v$ the edge map at the endpoint $v$, that is, the restriction of $f_e$ on the corresponding endpoint $0$ or $1$.
\end{enumerate}
We say it satisfies the {\em quasi-isometrically embedded condition} if $f_{e,v}$ is a quasi-isometric embedding uniformly for each edge $e$ of $T$ and an endpoint $v$ of $e$.

We say it is a \emph{tree of relatively hyperbolic spaces} if each $\Xc_v$ and $\Xc_e$ is a relatively hyperbolic space with respect to some collections of closed subsets $\Hc_v$ and $\Hc_e$ respectively, and their constants for hyperbolicity and separation are uniform. In this case, we say $f_{e,v}$ is \emph{strictly type-preserving} if
\begin{enumerate}
    \item For every $H\in\Hc_e$, there is $H'\in\Hc_v$ such that $f_{e,v}(H)\subset H'$.
    \item For every $H'\in\Hc_v$, the set $f_{e,v}^{-1}(H')$ is either empty or in $\Hc_e$.
\end{enumerate}

Assume that all endpoint maps are strictly type-preserving. Consider a  graph $T_0 = (V(T_0),E(T_0))$ consisting of the following data. Each vertex in $V(T_0)$ is a pair $(v,H_v)$ for some vertex $v$ of $T$ and $H_v\in\Hc_v$ and each edge in $E(T_0)$ is a pair $(e,H_e)$ for some edge $e$ and $H\in\Hc_e$. If $e$ has endpoints $v,w$, the edge $(e,H_e)$ joins $(v,H_v)$ to $(w,H_w)$ if $f_{e,v}(H)\subset H_v$ and $f_{e,w}(H)\subset H_w$. Suppose $T_0$ is such a maximal connected graph (with respect to inclusion) as described above. We call \[C_{T_0} = \bigl( \bigcup_{(v,H_v)\in V(T_0)} H_v \bigr)\cup \bigl(\bigcup_{(e,H)\in E(T_0)} f_e(H_e\times[0,1])\bigr)\subset \Xc\] a \emph{maximal cone-subtree of horosphere-like sets}. We denote by $\Cc$ the collection of all the maximal cone-subtrees of horosphere-like sets of $\Xc$.

Note also that when all endpoint maps are strictly type-preserving. Each of them extends to a map of coned-off spaces $\widehat f_{e,v}:\widehat \Xc_e\to\widehat \Xc_v$, by $c_H\mapsto c_{H'}$. Therefore $\mathrm{Pr}:\Xc\to T$ gives a \emph{induced tree of coned-off spaces} $\widehat{\mathrm{Pr}}: \widehat \Xc \to T$.
\end{definition}

\begin{theorem}[{\cite[Theorem~2.8]{mj2011relative}}]\label{theorem: cannon--thurston map on bowditch boundaries}
    Let $\mathrm{Pr}: \Xc \to T$ be a tree of relatively hyperbolic spaces satisfying the quasi-isometrically embedded condition, the inclusion of edge spaces into vertex spaces is strictly type-preserving, and the induced tree of coned-off spaces satisfies the quasi-isometrically embedded condition. If $\Xc$ is hyperbolic relative to the family $\Cc$ of maximal cone-subtrees of horosphere-like sets, then a Cannon-Thurston map $\partial_\infty \Xc_v^h \to \partial_\infty \Xc^h$ exists for the proper embedding $i : \Xc_v \to \Xc$ for any vertex $v$ of $T$. 
\end{theorem}

Suppose $Y_0$ is a geodesic metric space and $\Hc_0$ is a collection of closed subsets whose path metrics are geodesic. We \emph{thicken} the pair $(Y_0,\Hc_0)$ by the following process. For each $H\in \Hc_0$, we attach $H\times [0,1]$ to $H$ in $Y_0$ at $H \times \{0\}$. We denote the resulting space by $\Tc(Y_0,\Hc_0)$, let \[\Hc(Y_0,\Hc_0) = \{H\times \{1\} \subset \Tc(Y_0,\Hc_0): H\in \Hc_0\}~.\] Then $(Y_0,\Hc_0)$ is quasi-isometric to $(\Tc(Y_0,\Hc_0),\Hc(Y_0,\Hc_0))$ by the embedding. The purpose of considering the thickening is to ensure the horosphere-like sets are uniformly separated.

Suppose $(\Gamma,\Pc)$ is a relatively hyperbolic group. We fix a finite adapted generating set $S$ of $\Gamma$, that is, $S$ is symmetric and $S\cap P$ generates $P$ for every $P\in\Pc$. We denote by $\Cay(P)$ the Cayley graph of $P$ with respect to the generating set $P\cap S$ for each $P\in \Pc$ and set \[\Hc_0(\Gamma,\Pc) = (\{(\gamma \Cay(P)): \gamma\in \Gamma, \ P\in\Pc\}~.\] We briefly denote the thickening of $(\Cay(\Gamma),\Hc_0(\Gamma,\Pc))$ by $\Tc(\Gamma,\Pc)$. Then $\Tc(\Gamma,\Pc)$ relatively hyperbolic spaces. Actually, $\Tc(\Gamma,\Pc)^h$ is a Gromov model of $(\Gamma,\Pc)$ that are quasi-isometric to the so-called Groves--Manning cusp space introduced in \cite{GM08}. In particular, $\partial_\infty \Tc(\Gamma,\Pc)^h = \partial_\infty(\Gamma,\Pc)$.

We then show the example we constructed in Section~\ref{section: geometrically infinite example} admits a Cannon--Thurston map. Recall that $\Gamma$ is a free group of two generators $\alpha,\beta$, and $P$ is the cyclic subgroup of $\Gamma$ generated by $\alpha$. Then $(\Gamma,P)$ is a relatively hyperbolic group by Theorem~\ref{theorem: conditions on relative hyperbolicity}. We fix the adapted generating set $\{\alpha^{\pm 1}, \beta^{\pm 1}\}$ of $\Gamma$. Recall also that for a fixed large enough constant, $\phi: \Gamma\to \Gamma$ is defined to be the injective endomorphism defined by $\phi(\alpha) = \alpha^m$ and $\phi(\beta) =\beta \alpha\beta \alpha^2\dots \beta\alpha^{m-1}\beta$. The group \[G = \langle \alpha,\beta,t:t\alpha t^{-1}=\phi(\alpha),\ t\beta t^{-1}=\phi(\beta)\rangle\] is hyperbolic relative to $Q = \langle \alpha,t:t\alpha t^{-1}=\phi(\alpha) \rangle \subset G$ by Theorem~\ref{theorem: 1/6 embedding factor} and Theorem~\ref{theorem: 1/6 relatively hyperbolic}. We fix the adapted generating set $\{\alpha^{\pm 1}, \beta^{\pm 1},t^{\pm 1}\}$ of $G$.

Let $\rho:\Gamma\to G$ be the representation defined by $\rho(\alpha)=\alpha$ and $\rho(\beta)=\beta$.

\begin{proposition}\label{proposition: existence of Cannon--Thurston map for the example}
    The representation $\rho$ admits a Cannon--Thurston map \[\partial_\infty(\Gamma,P)\cong \partial_\infty \Tc(\Gamma,P)^h \to \partial_\infty \Tc(G,Q)^h \cong \partial_\infty(G,Q)~.\]
\end{proposition}

\begin{proof}
    We extend $\phi:\Gamma\to \Gamma$ as a map on $\Cay(\Gamma)$ by mapping each edge onto the geodesic segment joining the images of its endpoints, then extend the map further on $\Tc(\Gamma,P)$ by mapping each point $(x,r)$ on $\gamma \Cay(P)\times [0,1]$ to $(\phi(x),r)$ for each $\gamma\in \Gamma$. We denote the extended map by $\Phi:\Tc(\Gamma,P)\to \Tc(\Gamma,P)$. Note that $\phi$ is a quasi-isometric embedding and therefore $\Phi$ is also a quasi-isometric embedding.

    For any left coset $g\Gamma\in G/\Gamma$, there is a unique left coset $g'\Gamma\in G/\Gamma$ such that $(g\Gamma)t^{-1}\subset g'\Gamma$. Actually, \[g\Gamma t^{-1} = gt^{-1}(t\Gamma t^{-1}) = gt^{-1} \phi(\Gamma)\subset gt^{-1} \Gamma~,\] and we may take $g'=gt^{-1}$ We define $\phi_{g\Gamma}: g\Gamma \to g'\Gamma$ by $\phi_{g\Gamma}(\gamma) = \gamma t^{-1}$ for any $\gamma\in g\Gamma$. Note that by definition, $\phi_{g\Gamma}$ is independent of the choice of representatives.
    
    For any choice of representatives $g,g'\in G$ such that $g\Gamma t^{-1} \subset g' \Gamma$, we have for any $\gamma\in \Gamma$, $\phi(g\gamma) = g\gamma t^{-1} = gt^{-1} \phi(\gamma) = g' (g'^{-1}gt^{-1}\phi(\gamma))$. Hence when identify both $g\Gamma$ and $g'\Gamma$ as copies of $\Gamma$, $\phi_{g\Gamma}$ is a composition of a left translation on $\Gamma$ and $\phi$. We briefly denote \[\Tc(g \Gamma, gP):=\Tc(g \Cay(\Gamma), \{g\gamma \Cay(P):\gamma\in \Gamma\})~,\] where we view $\Cay(\Gamma)$ as a subspace of $\Cay(G)$. Then $\phi_{g\Gamma}$ extends to a map $\Phi_{g\Gamma}: \Tc(g\Gamma,gP) \to \Tc(g'\Gamma,g'P)$ in the same way as we extend $\phi$, and when regard $\Tc(g\Gamma,gP)$ and $\Tc(g'\Gamma,g'P)$ as copies of $\Tc(\Gamma,P)$, $\Phi_{g\Gamma}$ is a left translation of $\Phi$. In particular, $\Phi_{g\Gamma}$ is identified with $\Phi$ up to isometries.
    
    Let $T$ be a graph with vertices labeled by $G/\Gamma$ and there is an edge $e_{g\Gamma}$ between a left coset $g\Gamma\in G/\Gamma$ and $gt^{-1}\Gamma$ the unique left coset that contains $g\Gamma t^{-1}$. We define the height of a vertex $g\Gamma$ to be the total exponent of $t$ in an (any) element of $g\Gamma$. Then each vertex has exactly one edge joining a vertex with height one less and hence $T$ contains no cycle. Note that the edge from a vertex to another vertex with height one more is not unique. Actually, they are labeled by elements in $t^{-1}\Gamma t/\Gamma$ and $ht^{-1}\Gamma=gt^{-1}\Gamma$ if and only if $g^{-1}h\in t^{-1}\Gamma t$.

    We construct a tree of relatively hyperbolic spaces by the following process. Assign each vertex $g\Gamma$ of $T$ the relatively hyperbolic space $\Tc(g \Gamma, gP)$. Note that each $\Tc(g \Gamma, gP)$ is a copy of $\Tc(\Gamma, P)$. Assign each edge $e_{g\Gamma}$ of $T$ the space \[\Tc(g \Gamma, gP) \times e_{g\Gamma}\cong \Tc(g \Gamma, gP) \times [0,1]\] where we identify $e$ with $[0,1]$ from $g\Gamma$ to $gt^{-1}\Gamma$. For each edge $e_{g\Gamma}$, identify $\Tc(g \Gamma, gP) \times \{0\}$ to $\Tc(g \Gamma, gP)$ by identity map, and identify $\Tc(g\Gamma, gP) \times \{1\}$ to $\Tc(gt^{-1}\Gamma, gt^{-1}P)$ by $\Phi_{g\Gamma}$. We denote the resulting space by $\Xc$. Let $\mathrm{Pr}:\Xc\to T$ be the projection such that each $\Tc(g \Gamma, gP)$ assigned to a vertex $g \Gamma$ is mapped to $g \Gamma$, and each point in $\Tc(g \Gamma, gP)\times e$ assigned to an edge $e$ is mapped to the second coordinate. Then $\mathrm{Pr}:\Xc\to T$ is a tree of relatively hyperbolic spaces. It satisfies the quasi-isometrically embedded condition since $\Phi$ is a quasi-isometric embedding and $\Phi_{g\Gamma}$ is identified with $\Phi$ up to isometries for all $g\Gamma\in G/\Gamma$.

    We set $H_P=\Cay(P)\times\{1\} \subset \Tc(\Gamma,P)$. Recall that $\{\gamma H_P: \gamma\in \Gamma\}$ is the collection of horosphere-like sets of $\Tc(\Gamma,P)$. Then one can easily verify that edge maps are all strictly type-preserving.
    
    Let $e_\Gamma$ be the edge from $\Gamma$ to $t^{-1}\Gamma$, and let $f_{e_\Gamma}: \Tc(\Gamma,P)\times[0,1]\to \Xc$ be the edge map on $e_\Gamma$. For each $gQ\in G/Q$, let \[C_{gQ} = \bigcup_{q\in Q} gq\cdot f_{e_\Gamma}(H_P\times[0,1])~.\] The definition is independent of the representative $g$ of $gQ$. Then \[\Cc=\{C_{gQ}:gQ\in G/Q\}\] is precisely the collection of maximal cone-subtrees of horosphere-like sets of $\Xc$.
    
    On the other hand, since $(G,Q)$ is relatively hyperbolic, the $G$-orbit of the set $H_Q = \Cay(Q)\times\{1\} \subset \Tc(G,Q)$ provides the horosphere-like sets of $\Tc(G,Q)$.

    The group $G$ acts properly and cocompactly on both $\Xc$ and $\Tc(G,Q)$. Moreover, $Q$ acts properly and cocompactly on $C_Q$ and $H_Q$, equipped with their intrinsic path metrics. Consequently, we can choose a $G$-equivariant quasi-isometry $F:\Xc\to \Tc(G,Q)$ such that $F(C_{gQ})\subset gH_Q$ and \[F\vert_{C_{gQ}}:C_{gQ}\to gH_Q \] is a quasi-isometry for the intrinsic path metrics. Then $F$ can extends over the attached horoballs to give a quasi-isometry $\Xc^h\to \Tc(G,Q)^h$. In particular $\partial_\infty \Xc^h \cong \partial_\infty \Tc(G,Q)^h\cong \partial_\infty (G,Q)$.

    By \cite[Theorem~10.1]{hruska2010relative}, the induced map on the coned-off space $\widehat \Phi: \widehat\Tc(\Gamma,P) \to \widehat\Tc(\Gamma,P)$ is also a quasi-isometric embedding. 
    We apply again the fact that $\Phi_{g\Gamma}$ is identified with $\Phi$ up to isometries for each $g\Gamma\in G/\Gamma$ to deduce that  the induced tree of coned-off spaces $\widehat{\mathrm{Pr}}: \widehat \Xc\to T$ satisfies the quasi-isometrically embedded condition.
    
    We have verified that $\mathrm{Pr}:\Xc\to T$ is a tree of relatively hyperbolic spaces satisfying the required conditions in Theorem~\ref{theorem: cannon--thurston map on bowditch boundaries}. Therefore, there exists a Cannon--Thurston map for the embedding $\Tc(\Gamma,P) \to \Xc$, that is, a $\rho$-equivariant map from $\partial_\infty \Tc(\Gamma,P)^h\to \partial_\infty \Xc^h$. 
\end{proof}

\bibliographystyle{alpha}
\bibliography{reference.bib}

\end{document}